\documentclass[10pt,reqno]{amsart}
\usepackage{mathtools}
\usepackage{color}
\usepackage{graphicx}
\usepackage{epstopdf}
\graphicspath{{figs/}}
\usepackage[hidelinks]{hyperref}
 \newtheorem{thm}{Theorem}[section]

\newtheorem{lem}[thm]{Lemma}

\newtheorem{rem}[thm]{Remark}

\numberwithin{equation}{section}

\newcommand{\abs}[1]{\left\vert#1\right\vert}

\newcommand{\set}[1]{\left\{#1\right\}}
\newcommand{\Real}{\mathbb R}

\begin{document}

\title[Generalized multiquadric radial function for integral fractional Laplacian]{Fractional-order dependent  Radial basis functions meshless methods for the integral  fractional Laplacian}


\author{Zhaopeng Hao}
\address{School of Mathematics, Southeast University, Nanjing 210096, PR China}
\email{zphao@seu.edu.cn}

\author{Zhiqiang Cai}
\address{Department of Mathematics, Purdue University, 150 N. University Street, West Lafayette, IN 47907}
\email{caiz@purdue.edu}

\author{Zhongqiang Zhang}
\address{Department of Mathematical Sciences, Worcester Polytechnic Institute, Worcester, Ma 01609 USA}\email{zzhang7@wpi.edu}




\date{\today}


\date{\today}

\begin{abstract}
We study the numerical evaluation of the integral fractional Laplacian and its application in solving fractional diffusion equations.
We derive a pseudo-spectral formula for the integral fractional Laplacian operator based on fractional order-dependent,  generalized multi-quadratic radial basis functions (RBFs) to address efficient computation of the hyper-singular integral.
 We apply the proposed formula to solving fractional diffusion equations and design a simple, easy-to-implement and nearly integration-free meshless method.
 We discuss the convergence of the novel meshless method through equivalent  Galerkin formulations.  
   We carry out numerical experiments to demonstrate the accuracy and efficiency of the proposed approach compared to the existing method using Gaussian RBFs. 
\end{abstract}

\maketitle


\section{Introduction }

Fractional Laplacian is a powerful mathematical tool in modeling the accurate word physical phenomenon such as anomalous \cite{MK00}, electromagnetism \cite{Bonito-Peng-2020}, quasi-geostrophic flows \cite{Bonito-Nazarov-2021}, and spatial statistics \cite{Lindgren-Rue-Lind-2011}.
Extensive numerical methods for 
fractional Laplacian have been developed over the past decade, see e.g. review papers 
\cite{Bonito-Bor-Noche-2018,Lischke17,Marta-Du-Gunzburger-2020,HarLM20}. However,  their computation is still very challenging due to the nonlocality and singularity.  After these comprehensive review papers,  more papers are focused on the fast evaluation of the fractional Laplacian and treatment of singularity, 
such as finite difference methods 
\cite{Hao-Zhang-Du-2021} on a square, spectral methods \cite{Hao-Li-Zhang-Zhang-2021}  on a disk, global sinc basis or Hermite function methods 
\cite{Antil-Sinc-2021,FlaVG21}, and finite volume  methods \cite{XuCLQ20}, Gaussian radial basis function methods \cite{Burkard-Wu-Zhang-2021}.

However,  the evaluation of fractional Laplacian on general domains still requires significant computational efforts.  Except in \cite{Burkard-Wu-Zhang-2021}, all the numerical methods above require numerical integration over hypersingular kernels or differentiation of singular functions, and thus the assembly of the differentiation matrix
is rather expensive.  In  \cite{Burkard-Wu-Zhang-2021},
the authors use the fact that the fractional Laplacian (defined via a hypersingular integral) of the Gaussian radial function has an analytic representation via   confluent hypergeometric functions. Yet the computation burden lies on the computation of this special function, which does not significantly reduce the computational cost, especially when a moderate or large scale of numerical evaluation is needed. 

\subsection*{Contributions}
In this work, we propose to use the generalized multiquadric functions 
$(1+\abs{x}^2)^\beta$ ({\color{black} $\beta=(\alpha-d)/2$, properly chosen. Here $\alpha$ is the order of the fractional Laplacian, and $d$ is the space dimension.}) as the radial basis functions, which leads to explicit and simple evaluations of integral fractional Laplacian ({\color{black}Theorem \ref{RBF-Lemma-pseudo-spectral-relation}}). Specifically, the integral fractional Laplacian of  a generalized multiquadric function  with proper indices is still a
generalized multiquadric function.
As a consequence, numerically evaluating the integral fractional Laplacian of the generalized multiquadric functions (GMQ, see e.g., \cite{Buhmann-2003,Wendland-2005,Fasshauer-2007,Fornberg-Flyer-2015} for the definition of GMQ) is much faster than that of Gaussian radial functions. In Table \ref{table:cpu}, we present the evaluations of 
these two radial functions at $N$ uniformly distributed equispaced  points  on the interval $(-1,1)$.  
%
%
%

\begin{table}[!htb]
	\centering
	\caption{Comparison of integral fractional Laplacian in the CPU time measured in seconds between hypergeometric functions in \cite{Burkard-Wu-Zhang-2021} and the proposed multi-quadratic function concerning the total number equispaced points on $(-1,1)$. 
		All the tests throughout the paper are  carried out  with  MATLAB R2020a and
		on a PC with Processor	Intel(R) Core(TM) i5-1035G7 CPU @ 1.20GHz. }\label{table:cpu}
	\begin{tabular}{c|cccc}  
		\hline \hline
		& $N=1024$ & $N= 2048$ &$N= 4096$ & $N=8192$ \\
		\hline
		GMQ	&   4.3e-04      & 6.0e-4 & 3.33-4 & 5.2e-4\\ 
		hypergeometric	& 1.96 & 4.08 & 9.06 & 19.8 \\
		\hline \hline
	\end{tabular}
\end{table}

Based on the proposed radial basis functions, we develop efficient meshless numerical methods for the fractional Laplacian {\color{black}(with order $\alpha\neq d$)} and fractional diffusion equations. 
As the fractional Laplacian of the proposed radial basis functions is exact and requires little effort to evaluate, our method is stable for a relatively large number of collocation points. 
Also, the method enjoys spectral accuracy if  the solution  is smooth. Specifically, our method benefits from the pseudo-spectral formula derived in \ref{subsection-pseudo-spectral}.
Compared to similar meshless pseudo-spectral method using different RBFs, such as Gaussian functions, our method is simple to implement and easy to combine with available fast solvers, e.g. \cite{Xia-2021}. 


Here we summarize the goal and novelty of this work. 
We explore a new meshless method for solving the PDEs with the integral fractional Laplacian. The novelty of this work is that we present a new  method avoiding singular integration and series truncation to the fractional Laplacian on the general bounded domain.
We also discuss the convergence of the method and demonstrate the method has spectral convergence.

We organize the rest of the paper as follows. 
In Section \ref{sec:frac-lap-gmq}, we derive the pseudo-spectral formula for the integral fractional Laplacian
of a special generalized multi-quadratic function. We also present the collocation method and some details of calculations of the method.
In Section \ref{section-PSF-GF},
we prove the  well-posedness of the considered problem and 
prove the convergence of Galerkin methods equivalent to the collocation method.  We first describe  conforming and nonconforming Galerkin formulations that are equivalent to
the collocation methods. We then discuss the convergence of the Galerkin formulations.
In 
Section \ref{section-collocation-reformulation-implementation},
we show  the equivalence of the collocation formulations with Galerkin formulations with numerical integration.   
%
We present in
Section \ref{section-Numerical-experiments} numerical results regarding the accuracy and the condition numbers. We also apply our approach to the time-dependent problem with mixed diffusion. Lastly, we summarize our findings and discuss some computational issues for future research.

\section{Fractional Laplacian of some generalized multi-quadratic functions}
\label{sec:frac-lap-gmq}

In this section, we derive the key pseudo-spectral formula for GMQ RBFs and present a collocation method based on 
the formula.

\subsection{Derivations}\label{subsection-pseudo-spectral}

For $x=(x_1,x_2,\cdots,x_d) \in \mathbb{R}^d$,   we use $|x|$ to denote its Euclidean norm.

The  fractional Laplacian is defined as the  singular integral, 
\begin{equation}\label{eq:fracLap-v1}  
	(-\Delta)^{\alpha/2}u(x) := c_{ d,\alpha}  \int_{\mathbb{R}^d } \frac{u(x) - u(y)} 
	{|x-y|^{d+\alpha}} \, dy,  \qquad \ 
	c_{d,\alpha} := \frac{2^\alpha\Gamma(\frac{\alpha+d}{2})}{\pi^{d/2} \abs{\Gamma(-\alpha/2)}},
\end{equation} 
where $\Gamma (\cdot)$ stands for the gamma function.  If $\alpha \geq 1$, the integrand is not integrable in the Lebesgue sense, and integration  is understood in the  sense of Cauchy  principal value;  see \cite{Samko-2002},  
One of the most important properties  of the fractional Laplacian is its Fourier transform, 
\begin{eqnarray}\label{fourier-property-laplace}
	(-\Delta)^{\alpha/2}u(x)=\mathcal{F}^{-1}[|\omega|^{\alpha}\mathcal{F}(u)],
\end{eqnarray}
where $\mathcal{F}$ denotes the standard Fourier transform. This property is often used to be one of equivalent definitions for the fractional operator. Based on the property 
\eqref{fourier-property-laplace},
by the translation and scaling properties of Fourier transform, it is not hard to prove the following useful properties; see also \cite{SamkoKM1993}. 

\begin{lem}[Translation and Scaling Properties]
	For function $v$, assume that $\mathcal{V}(x):=(-\Delta)^{\alpha/2} v(x)$ is well defined for $x\in \mathbb{R}^d$. Then it satisfies the following properties: 
	\begin{eqnarray}\label{translation-property}
		(-\Delta)^{\alpha/2}[v(x-x_0)]= \mathcal{V}(x-x_0),
	\end{eqnarray}
	for any fixed point $x_0\in \mathbb{R}^d$, and 
	\begin{eqnarray}\label{scaling-property}
		(-\Delta)^{\alpha/2}[v(c x)]=|c|^{\alpha} \mathcal{V}(cx),
	\end{eqnarray}
	for any scalar $c\in \mathbb{R}$. 
\end{lem}

The following lemma plays a pivotal role in deriving the pseudo-spectral formulas in this work. 
\begin{lem}
	The  radially symmetrical function $(1+ |x|^2)^{\beta}$, $x\in \mathbb{R}^d$, with  {\color{black}   $\beta \in \mathbb{R}/\mathbb{N}_0$} possesses the Fourier transform
	\begin{eqnarray}\label{fourier-property-inverse-quadratic}
		\mathcal{F}\big[(1+ |x|^2)^{\beta}\big](\omega)=\frac{2^{1+\beta}}{\Gamma(-\beta)} |\omega|^{-\beta-d/2} K_{d/2+\beta}(|\omega|),\quad \mathbf{\omega} \neq 0.
	\end{eqnarray}
	Here $K_{\beta}$ is the modified Bessel function of the third kind of order $\beta \in \mathbb{C}$. Moreover,  for any $\beta \in \mathbb{R}$, it holds that 
	\begin{eqnarray}\label{Laplacian-property-inverse-quadratic}
		-\Delta \big[(1+ |x|^2)^{\beta}\big]=-[2{\color{black}d}\beta +4\beta (\beta-1)] (1+|x|^2)^{\beta-1} +4\beta (\beta-1) (1+|x|^2)^{\beta-2} .
	\end{eqnarray}
\end{lem}
\begin{proof}
	The formula \eqref{fourier-property-inverse-quadratic} is	{\color{black}directly from}  \cite[Page 109, Theorem 8.15]{Wendland-2005}.  The direct calculation leads to \eqref{Laplacian-property-inverse-quadratic}. 
\end{proof}

Based on the above lemma, we can derive the following formula, the pseudo-spectral relation between the inverse multi-quadratic RBFs. 
\begin{thm} \label{RBF-Lemma-pseudo-spectral-relation}
	When  {\color{black}$d-\alpha >0$ or} $d-\alpha $ is not even number, we have 
	\begin{eqnarray}\label{frac-lap-speudo-spectral-relation}
		(-\Delta)^{\alpha/2} \bigg( \Phi_{d,\alpha}(x)  \bigg)= \mu_{d,\alpha} \Phi_{d,-\alpha}(x),
	\end{eqnarray}
	where
	\begin{eqnarray}
		\Phi_{d,\alpha}(x): = (1+ |x|^2)^{\alpha/2-d/2},\quad   \mu_{d,\alpha}=\frac{2^{\alpha}\Gamma(d/2+\alpha/2) }{\Gamma(d/2-\alpha/2)}.
	\end{eqnarray}
\end{thm}
\begin{proof}
{\color{black}Denote  $\beta=(\alpha-d)/2$.} 	It suffices to verify that the Fourier transforms on both sides of \eqref{frac-lap-speudo-spectral-relation} are equal. Using the properties \eqref{fourier-property-laplace}, \eqref{fourier-property-inverse-quadratic} together with the fact $K_{\beta}= K_{-\beta}$ leads to the desired result.  {\color{black}For $\beta<0$, see the derivation by \cite{Huang-2014-gmq}. }
\end{proof}

This formula is the starting point of our meshless method. We will use the formula \eqref{frac-lap-speudo-spectral-relation} to design a singular integration-free method for the hyper-singular integral fractional Laplacian. 
Though the dimension number $d$ can not be equal to the order $\alpha$ in the formula, the conditions are not restrictive in most applications, as we are usually interested in multi-dimension problems with  $d\geq 2$ and $\alpha<2$. 

{\color{black}
\begin{rem}
By the semigroup property of the fractional Laplacian operator, we have  
	\begin{eqnarray}\label{samko-2001-formula}
		(-\Delta)^{1-\alpha/2} \bigg(\Phi_{d,-\alpha}(x)\bigg)=(-\Delta)^{1-\alpha/2}    \bigg[\frac{1}{\mu_{d,\alpha}}(-\Delta)^{\alpha/2} \bigg(\Phi_{d,\alpha}(x)\bigg)\bigg]= \frac{1}{\mu_{d,\alpha}}(-\Delta) \bigg(\Phi_{d,\alpha}(x)\bigg).
	\end{eqnarray}
	Then using the formula \eqref{Laplacian-property-inverse-quadratic} with $\beta= \alpha/2-d/2$, we obtain that 
	\begin{eqnarray}
		(-\Delta)^{1-\alpha/2} \bigg(\Phi_{d,-\alpha}(x)\bigg)= \frac{2^{1-\alpha} \Gamma(d/2-\alpha/2+1)}{\Gamma(d/2+\alpha/2)}  \bigg[(\alpha-2)\Phi_{d,\alpha-2}(x)+ (d+2-\alpha)\Phi_{d,\alpha-4}(x)\bigg].
	\end{eqnarray}
	Replacing the order $2-\alpha$ by $\alpha$, we get another  formula in this work. That is,	for $\alpha<d+2$, we have 
	\begin{eqnarray}
		(-\Delta)^{\alpha/2} \bigg(\Phi_{d,\alpha-2}(x)\bigg)= \eta_{d,\alpha}^{(1)} \Phi_{d,-\alpha}(x)+ \eta_{d,\alpha}^{(2)}\Phi_{d,-\alpha-2}(x),
	\end{eqnarray}
	where 
	$$ \eta_{d,\alpha}^{(1)}=- \frac{\alpha 2^{\alpha-1} \Gamma(d/2+\alpha/2)}{\Gamma(d/2-\alpha/2+1)} ,\quad \eta_{d,\alpha}^{(2)}=  \frac{2^{\alpha} \Gamma(d/2+\alpha/2+1)}{\Gamma(d/2-\alpha/2+1)}.$$
	Such pseudo-spectral relation   may allows the computation of the case     $d = \alpha$. However, we haven't investigated the performance of the shape function $\Phi_{d,-\alpha}(x)$ in the numerical experiments. 


 

We also notice that	a similar formula  \eqref{samko-2001-formula} has been derived  by a different approach; see \cite[Page 324]{Samko-2002}.
\end{rem}
}

\subsection{Collocation methods for fractional Poisson equations}

Let $\Omega$ {\color{black}be}
a bounded and open domain in $\mathbb{R}^d$ with {\color{black} its  complement  and boundary  denoted by $ \Omega^c$ }and $ \partial \Omega$, respectively. 
Consider the following problem subject to the non-homogeneous boundary condition
{\color{black}\begin{eqnarray}\label{eq:frac-poisson-nonhomegeous}
	&& (-\Delta)^{\alpha/2} \tilde{u} = \tilde{f}, \quad  x \in  \Omega, \\
	&&  \tilde{u}=g, \quad x\in  \Omega^c. \notag
\end{eqnarray}
Denote the extension of $g(x)$ to 
$\Real^d$ by  $g(x)$. Let $u(x)=\tilde{u}(x)-g(x)$   and $f= \tilde{f} +  (-\Delta)^{\alpha/2}_{\Omega^c}g $  where   \begin{eqnarray*}
	&& (-\Delta)^{\alpha/2}_{\Omega^c}g(x) := 
	c_{d,\alpha}\int_{\Omega^c}\frac{g(y)} 
	{|x-y|^{d+\alpha}} \, dy.
\end{eqnarray*}	Then $u(x)$ solves }
\begin{eqnarray}\label{eq:frac-poisson}
	&& (-\Delta)^{\alpha/2} u= f, \quad  x \in  \Omega, \\
	&& {\color{black} u=0, \quad x\in  \Omega^c.} \notag
\end{eqnarray}

Let $X_{\Omega}=\{x_i\in \Omega \}$,  $X_{\partial\Omega}=\{x_i\in \partial\Omega \}$   with the points all assumed to be distinct. Let   $X_{\bar{\Omega}}= X_{\Omega} \cup X_{\partial\Omega} $. 
The mesh norm for $X$ relative to $\Omega$ is $h=h_{X,\Omega}:=\sup_{x\in \Omega}\min_{x_j\in X}|x-x_j| $, which measures the maximum distance of  any points in $\Omega$ can be from $X_{\Omega}$. The other useful length is the separation radius, 
$q=q_{X,\Omega}:= \frac{1}{2} \min_{j\neq k} |x_j-x_k|$. It is clear that $h_{X,\Omega} \geq q_{X,\Omega}$; the equality can hold only for a uniform distribution of points on an interval. The mesh ratio $\rho=\rho_{X,\Omega}:= h/q \geq 1 $ measures how uniformly points in $X$ are distributed in $\Omega$. If $\rho$ is small, we say it is quasi-uniformly distributed.   
The collocation method for \eqref{eq:frac-poisson} can be described as follow: at the scattered points $X_{\Omega}$ and the boundary $X_{\partial\Omega}$
\begin{eqnarray}
	&& (-\Delta)^{\alpha/2} u(x_i) = f(x_i), \quad  x_i \in  X_{\Omega}, \\
	&& u(x_i)=0, \quad x_i\in  X_{\partial \Omega}. 
\end{eqnarray}
Let the approximated solution $u_N(x)= \sum_{j=1}^N\lambda_j \phi_j(x)$, and
replace  the  exact  solution $u$ by the approximated one $u_N$.  {\color{black} Denote $N$ and $N_o$ the total number of points associted with the collocation points set $X_{\bar{\Omega}}$ and $X_{\Omega}$, respectively.} Using the pseudo-spectral relation in  Theorem \ref{RBF-Lemma-pseudo-spectral-relation} leads to the linear system  
\begin{eqnarray}
	&&\sum_{j=1}^N \lambda_j \bigg(\varepsilon^\alpha  \mu_{d,\alpha} \psi_j(x_i) + (-\Delta)^{\alpha/2}_{\Omega^c} \phi_j \bigg) =f(x_i), \quad  1\leq i\leq N^o,  \label{eqn:discrete-interior}\\ 
	&& \sum_{j=1}^{N}\lambda_j \phi_j(x_i) =
	0,\quad  \quad N^o+1\leq i\leq N, \label{eqn:discrete-boundary}
\end{eqnarray}
where 	$\phi_j(x) = ( \varepsilon^2+| x-x_j|^2)^{\alpha/2-d/2},\, \psi_j(x) = ( \varepsilon^2+| x-x_j|^2)^{-\alpha/2-d/2}
$ are from Theorem \ref{RBF-Lemma-pseudo-spectral-relation} .
%
%
The collocation points are  quasi-uniformly distributed. In Section \ref{section-PSF-GF}, we use some quadrature points such that the collocation method above is equivalent to a Galerkin method with numerical integration, where we can establish convergence of the method.


{\color{black}
	
	\begin{rem}\label{rem-reformulation-of-the-scheme}
Denote  $	X_{D/\Omega}=\{x_i \in \Omega^c \cap D  \} 
$ and  $X_{D}= X_{\Omega} \cup X_{D/\Omega} $. Assume that the bounded domain $\Omega$ can be embedded in the unit disk $D$. Otherwise, we can use the scaling properties stated in Lemma  \ref{scaling-property} to equivalently reformulate the problem. 
To implement  the scheme \eqref{eqn:discrete-interior}-\eqref{eqn:discrete-boundary}, we need to compute the integral  $(-\Delta)^{\alpha/2}_{\Omega^c} \phi_j (x)$ for $x\in \Omega$. The strategy is to split the domain $\Omega^c$ into the two parts, the bounded one plus the unbounded one.   That is $\Omega^c= D/\Omega \cup D^c$ or $(-\Delta)^{\alpha/2}_{\Omega^c}= (-\Delta)^{\alpha/2}_{D/\Omega} + (-\Delta)^{\alpha/2}_{D^c}$ .  We will describe the computation of the second part $(-\Delta)^{\alpha/2}_{D^c} $ in the next section.  For the integration over the bounded part  $ D/\Omega$, we can use the standard high-accuracy quadrature rule to compute the regular integral. If the boundary of the domain $\Omega$ is very complex, however, it may increase the overhead of computation to handle the complex geometry.  To keep the simplicity of collocation method, it suffices to consider  collocation points in the larger domain,  the unit disk $D$ instead of $\Omega$, and associate the total number $N$ with the collocation points set $X_D$.  In this case, the operator  $(-\Delta)^{\alpha/2}_{\Omega^c}$ associates with the complex domain $\Omega^c$ in the scheme  \eqref{eqn:discrete-interior} is replaced by $ (-\Delta)^{\alpha/2}_{D^c}$. 
	\end{rem}

}

\subsection{Implementation}	
\label{subsuction-complement-domain}


Here we provide the computation of $(-\Delta)^{\alpha/2}_{\color{black}D^c} \phi_j (x_i)$, denoted as $(B_\phi)_{i,j}$, the integral on the complement of the unit interval  in 1D or disk in  2D. 
An equivalent implementation of \eqref{eqn:discrete-interior}-\eqref{eqn:discrete-boundary} will be discussed in Section \ref{section-collocation-reformulation-implementation}.

The scattered data points/centers $x_i$ are located in the  interior of  $\Omega$ to avoid the singularity. 
Moreover,  when  the solution in the interior domain is of primary interest, it suffices to choose the centers inside the domain.
To have a square linear system, the number of centers is  the same as that of the scattered data points or the collocation points.

In 1D, we have  
\begin{eqnarray}
	(B_{\phi})_{i,j}=	\int_{1}^{\infty} \frac{\big(\varepsilon^2+|y-x_j|^2\big)^{\alpha/2-1/2}} 
	{(y-x_i)^{1+\alpha}} \, dy + \int_{-\infty}^{-1} \frac{\big(\varepsilon^2+|y-x_j|^2 \big)^{\alpha/2-1/2}} 
	{(x_i-y)^{1+\alpha}} \, dy=:	\mbox{I}_{i,j}+ \mbox{II}_{i,j}.
\end{eqnarray}
It suffices to consider the integral of $\mbox{I}_{i,j}$. Making the change of variable $y=1/s$ gives 
\begin{eqnarray}
	\mbox{I}_{i,j} &=& 	\int_{1}^{\infty} \frac{\big(\varepsilon^2+(y-x_j)^2\big)^{\alpha/2-1/2}} 
	{(y-x_i)^{1+\alpha}} \, dy \nonumber\\
	&=&  \int_{0}^{1} \big(s^2\varepsilon^2+(1-x_j s)^2\big)^{\alpha/2-1/2} 
	(1-x_i s)^{-1-\alpha} \, ds \notag.
\end{eqnarray}
Then using Gauss-Jacobi  quadrature with the Jacobi index pair $(0,0)$ associated with the interval $(0,1)$, we have 
\begin{eqnarray}
	\mbox{I}_{i,j}
	&\approx &   \sum_{k=1}^K w_k b_{i,k}c_{k,j},\quad b_{i,k}=
	(1-x_i s_k)^{-1-\alpha}, \quad  c_{k,j}= \big(s^2_k\varepsilon^2+(1-x_j s_k)^2\big)^{\alpha/2-1/2}. \label{def-matrix-B-C}
\end{eqnarray}
Analogously, for the second integral $\mbox{II}$, we have 
\begin{eqnarray}
	\mbox{II}_{i,j}
	&\approx &  \sum_{k=1}^K w_k b^\prime_{i,k}c^\prime_{k,j},\quad b^\prime_{i,k}=
	(1+x_i s_k)^{-1-\alpha}, \quad  c^\prime_{k,j}= \big(s^2_k\varepsilon^2+(1+x_j s_k)^2\big)^{\alpha/2-1/2}. 
\end{eqnarray}
Here we can see that  $b_{i,k}=b^{\prime}_{N+1-i,k}$ and $c_{k,j}=c^{\prime}_{k,N+1-j}$ for $1<i,j<N$.

In the computation, the number of grid points $N$ is far larger than the integration quadrature number $K$.  Thus, for the matrix $B_{\phi}$, we have the low-rank decomposition $B_{\phi}=BDC+B^\prime D C^\prime $. 
{\color{black} Here  $B= (b_{ij}) $ is a $N$ by $K$ and $C=(C_{ij})$ is a $K$ by $N$ matrix with the entries defined from (2.19) and $B'= (b'_{ij})$ and $C'= (c'_{ij})$ are from (2.20).  $D$ is a $K$ by $K$ diagonal matrix with entries $D_{k,k}=w_k$. In the numerical experiments in this work, $K$ is taken around $10$. }

Next, we consider the  2D case. Note that $x_i=(x_{i,1},x_{i,2})$,  $x_j=(x_{j,1},x_{j,2})$ and $y=(y_1,y_2) $. Making the polar coordinate transform $y_1=r\cos(\theta) $ and $y_2= r\sin(\theta)$, we obtain that 
\begin{eqnarray}
	&&(B_{\phi})_{i,j}= \int_{|y|>1} \frac{\big(\varepsilon^2+|y-x_j|^2\big)^{\alpha/2-1}} 
	{|x_i-y|^{2+\alpha}} \, dy \nonumber \\
	&=&	 \int_1^{\infty}\int_0^{2\pi}
	\frac{
		\big[\varepsilon^2+ (r\cos(\theta)-x_{j,1})^2+ (r\sin(\theta)- x_{j,2})^2 \big]^{\alpha/2-1}
	}{
		\big[(r\cos(\theta)-x_{i,1})^2+ (r\sin(\theta)- x_{i,2})^2 \big]^{\alpha/2+1} } rd\theta dr. 
\end{eqnarray}
Analogous to the 1D case, by change of variable $r=1/s$,  then we have  
\begin{eqnarray}
	(B_{\phi})_{i,j}
	=	 \int_0^{1}\int_0^{2\pi}
	\frac{
		\big[\varepsilon^2 s^2+ (\cos(\theta)-x_{j,1} s)^2+ (\sin(\theta)- x_{j,2} s)^2 \big]^{\alpha/2-1}
	}{
		\big[(\cos(\theta)-x_{i,1} s)^2+ (\sin(\theta)- x_{i,2} s)^2 \big]^{\alpha/2+1} } sd\theta ds.
\end{eqnarray}
Like the 1D case, we use the Gauss quadrature in the $s$ direction. For the $\theta$ direction,  since the integrand is analytical and periodic in the theta direction, we use the standard simple but accurate rectangular rule to approximate the integral. Precisely, we have 
\begin{eqnarray}
	(B_{\phi})_{i,j} 
	\approx \frac{2\pi}{M} \sum_{k=1}^K \sum_{m=1}^{M}
	\frac{
		\big[\varepsilon^2 s_k^2+ (\cos(\theta_m)-x_{j,1} s_k)^2+ (\sin(\theta_m)- x_{j,2} s_k)^2 \big]^{\alpha/2-1}
	}{
		\big[(\cos(\theta_m)-x_{i,1} s_k)^2+ (\sin(\theta_m)- x_{i,2} s_k)^2 \big]^{\alpha/2+1} } s_kw_k.
\end{eqnarray}
In the practical computation, the product $KM$ is far less than the total number $N$, the degree of freedoms. Thus we have the low-rank decomposition for the matrix $B_{\phi}$. 

\section{Stability and error estimates}
\label{section-PSF-GF}

\subsection{Stability in  Sobolev spaces} \label{section-preliminary}

We use the standard notation for Sobolev spaces with non-negative integer $m$, let 
\begin{eqnarray}
	H^m (\Omega)=\{v\in L^2(\Omega) : \|v\|_{H^m(\Omega)} < \infty \},	\quad  \|v\|_{H^m (\Omega)}= \bigg( \sum_{k=0}^m \sum_{|\beta|=k}\|D^{\beta} v \|^2_{L^2(\Omega)}  \bigg)^{1/2}  .
\end{eqnarray} 
Here  $\beta=(\beta_1,\cdots,\beta_d)$ is an multi-index and $|\beta|=\beta_1+\cdots+\beta_d$. 
%
For $0<s<1$, a fractional order Sobolev space $H^s(\Omega)$ \cite{Adams75} is defined by 
$$
H^{s}(\Omega) = \{v\in L^2(\Omega): \|v\|_{H^s(\Omega)} <\infty\}, \quad
\|v\|_{H^s(\Omega)} = \left( \| v\|_{L^2(\Omega)}^2 + |v|_{H^s(\Omega)}^2 
\right)^{1/2}, $$
where  the semi-norm  
\begin{equation}\label{Hs-norm-semi}
	|v|_{H^s(\Omega)} = \left(\iint_{\Omega\times \Omega} \frac{(v(x)-v(y))^2}{|x-y|^{d+2s}} dx dy\right)^{1/2}
\end{equation}

For $s>1$ let $s=m+\sigma$ with $m$ as the largest integer no greater than $s$ and with
$\sigma = s - m \in [0, 1)$.
Then 
$$
H^s(\Omega) = \{v\in H^m(\Omega): |D^{\beta}v| \in H^{\sigma} ~\mbox{for~ any}~ \beta~  
\mbox{such~ that}~ |\beta|=m \} 
$$
with the norm 
$$
\|v\|_{H^s(\Omega)} = \left( \|v\|_{H^m(\Omega)}^2 
+ \sum_{|\beta|=m} |D^{\beta} v|_{H^\sigma(\Omega)}^2 \right)^{1/2}.
$$
For every function $v\in L^2(\Omega)$, we denote by $\tilde{v}$ the extension by zero of $v$ to $\mathbb{R}$, i.e., $\tilde{v}(x)= v(x)$ if $x\in \Omega$ and $\tilde{v}(x)=0$ otherwise. 
Introduce the extension spaces to characterize the solution space for the fractional order problems.  Define $$\tilde{H}^s(\Omega):= \{v\in H^s(\Omega): \tilde{v} \in H^s(\mathbb{R}^d) \} $$ equipped with norm $ \|v\|_{\tilde{H}^s(\Omega)} := \| \tilde{v} \|_{H^s(\mathbb{R}^d)} $.
The zero trace space $H^s_0(\Omega)$ can be defined as the closure of $C_0^{\infty}(\Omega)$ 
with respective to the norm  of $\, H^s(\Omega)$. The following facts are standard: (e.g., see \cite[Page 33, Theorem 3.19]{Alexandre-Guermond-2021})
\begin{eqnarray}
	&&	H^s(\Omega) = H_0^s(\Omega) = \tilde{H}^s(\Omega),\quad  s<\frac{1}{2}; \\
	&& H^s(\Omega) = H_0^s(\Omega) \neq \tilde{H}^s(\Omega),\quad  s= \frac{1}{2};\\
	&& H^s(\Omega) \neq H_0^s(\Omega)= \tilde{H}^s(\Omega),  \quad s>\frac{1}{2} \mbox{~and~} s-\frac{1}{2}\neq \mathbb{N}.	
\end{eqnarray}
Let $s>0$. By using $L^2(\Omega)$ as a pivot space, we have that duality pairing between $H_0^s(\Omega)$ and its dual $H^{-s}(\Omega)= (H_0^s(\Omega)) ^{\prime} $.

\subsection{Well-posedness}

Consider the fractional diffusion   equation (FDE) on a bounded domain with the  following Dirichlet boundary condition,
\begin{eqnarray}
	&& \nu(-\Delta)u(x) + (-\Delta)^{\alpha/2} u(x) = f(x), \quad  x\in  \Omega,  \label{eqn:frac-Laplacian}\\
	&& u=0, \quad x \in  {\color{black}\Omega^c}, \label{eqn:frac-Laplacian-BC}
\end{eqnarray}
where $\nu\geq 0$,  $f(x)$ is a given  function. 

We define 
\begin{equation}\label{def-rho-regional-operator}
	\rho_{\Omega}(x) =c_{d,\alpha}\int_{\Omega^c}\frac{1}{|x-y|^{d+\alpha}}\,dy, \quad (-\Delta)_{\Omega}^{\alpha/2}v(x) := c_{ d,\alpha} \int_{\Omega} \frac{v(x) - v(y)} 
	{|x-y|^{d+\alpha}} \, dy,
\end{equation}
where $c_{d,\alpha}$ is defined in \eqref{eq:fracLap-v1}. 
For $v$ vanishing outside of the domain,  
\begin{eqnarray}\label{eq-conection}
	(-\Delta)^{\alpha/2} v(x)= (-\Delta)_{\Omega}^{\alpha/2} v(x)+ \rho_{\Omega}(x) v(x).	
\end{eqnarray}
Based on this connection, one can readily extend all the results discussed in this work to the regional operator. 

The variational  formulation of the problem \eqref{eqn:frac-Laplacian}-\eqref{eqn:frac-Laplacian-BC} is to find $u\in \tilde{H}^{1}(\Omega),$ such that 
\begin{eqnarray}\label{variational-formulation}
	a(u,v) := \nu(\nabla u, \nabla v) + 	((-\Delta)_{\Omega}^{\alpha/2}u,v)+ (\rho_{\Omega} u,v) = (f,v), \quad \forall v \in \tilde{H}^{1}(\Omega).
\end{eqnarray} 
%
%
The bilinear form  $a(\cdot, \cdot)$ is symmetrical, coercive and continuous \cite{AcoBor15,Mixed-diffusion-regularity-2022}.  
The wellposedness of the variational form is the direct consequence of Lax-Milgram's theorem. 
That is, if  $f\in {H}^{-1}(\Omega)$,  then there exists a unique 
solution $u\in \tilde{H}^{1}(\Omega)$ to the problem \eqref{variational-formulation}. If $\nu=0$, then there exists a unique 
solution $u\in \tilde{H}^{\alpha/2}(\Omega)$ to the problem \eqref{variational-formulation} when $f\in {H}^{-\alpha/2}(\Omega)$.

\begin{rem}
	One can readily extend the discussion throughout the paper to the case 
	of non-homogeneous boundary conditions by standard techniques. 
	It is possible to make a sufficiently smooth extension to the interior of domain $\Omega$.  
	%
	On sufficiently smooth domains, the regularity of the solution is dominated by both the right-hand side $f(x)$ and the exterior  data $g(x)$. 
	Higher regularity of the solution may be expected if the  $f(x)$ and  $g(x)$  are regular enough and compatible. 
\end{rem}
Throughout the following, for simplicity of the discussion, we only consider the fractional diffusion dominant model problem, $\nu=0$ in \eqref{eqn:frac-Laplacian}, unless otherwise stated.

\subsection{Convergence of an equivalent method}
Below we show the  convergence of a Galerkin method with numerical integration,  which is equivalent to the collocation method with selected collocation points. 
The equivalence of the two methods is shown in Section \ref{section-collocation-reformulation-implementation}.
We  establish the Galerkin method and its convergence. Assume that the numerical integration errors are negligible, the convergence order of the collocation method will be the same as the Galerkin reformulation.

\subsubsection{A  conforming Galerkin method for $\alpha\in (0,1)$}

Denote the trial space 
\begin{eqnarray}
	V_N (\Omega)= \mbox{Span} \{\phi_j(x), x_j\in X_{\Omega}\},
\end{eqnarray}
where  the  shape function  comes  from Theorem \ref{RBF-Lemma-pseudo-spectral-relation} and reads
\begin{eqnarray}
	\phi_j(x) = ( \varepsilon^2+| x-x_j|^2)^{\alpha/2-d/2},\quad x\in \Omega.  \label{def-shape-parameter-gmq}
\end{eqnarray} Here $\varepsilon>0$ is the \textit{shape parameter} that controls the flatness of the basis function. 
Note that
the RBFs in Theorem \ref{RBF-Lemma-pseudo-spectral-relation} can be used at the cost of double evaluation of the basis functions, for which the computational cost is negligible compared to other  RBFs, such as Gaussian ones.

The Lagrange function $L_i(x)$ centered at $x_i \in X_{\Omega}$ is defined to be the unique RBFs interpolant satisfying $L_i(x_j)=\delta_{i,j} $ for all $x_j \in X_{\Omega}$. It is obvious that 
\begin{equation}\label{conforming-subspace}
	V_N (\Omega)= \mbox{Span} \{L_j(x), {x_j\in X_{\Omega}}\}.
\end{equation}

When $\alpha<1$, recall that $H^{\alpha/2}(\Omega)= \tilde{H}^{\alpha/2} (\Omega) $.   The approximated solution space $V_N (\Omega)$ belongs to $\tilde{H}^{\alpha/2} (\Omega)$, i.e.,  $V_N (\Omega)\subset  \tilde{H}^{\alpha/2} (\Omega)  $.  With this fact and setup, we present the conforming  Galerkin method for the problem \eqref{eqn:frac-Laplacian}-\eqref{eqn:frac-Laplacian-BC} with $\nu=0$. 
The RBFs based Galerkin method is to find $u_N=\sum_{x_j\in {\Omega}}u_N(x_j) L_j(x) \in V_N (\Omega)$ such that 
\begin{eqnarray}\label{galerkin-scheme}
	a(u_N,v_N)=(f,v_N),\quad \forall v_N\in V_N (\Omega).
\end{eqnarray}

The well-posedness of discrete problem \eqref{galerkin-scheme} can be readily shown by Lax-Milgram's theorem, i.e.,  there exists a unique solution $u_N\in H^{\alpha/2}(\Omega)$ such that
\begin{eqnarray}
	\|u_N\|_{H^{\alpha/2}(\Omega)} \leq c \|f\|_{H^{-\alpha/2}(\Omega)}. 
\end{eqnarray}

To analyze the convergence of the conforming Galerkin method, we need the following lemma.
\begin{lem}
	Let $u$ and $u_N$ solves \eqref{eqn:frac-Laplacian}-\eqref{eqn:frac-Laplacian-BC} and \eqref{galerkin-scheme} respectively, then it holds that 
	\begin{eqnarray}
		\|u-u_N\|_{H^{\alpha/2}(\Omega) }\leq C \inf_{v_N\in V_N (\Omega)} \|u-v_N\|_{H^{\alpha/2}(\Omega)}
	\end{eqnarray}
\end{lem} 
\begin{proof}
	Subtracting \eqref{galerkin-scheme} from \eqref{variational-formulation} and using the coerciveness of bilinear form $a(\cdot,\cdot)$ lead to the desired results immediately. 
\end{proof}


For simplicity, we assume these approximation properties hold
\begin{eqnarray}\label{rbf-approximation-properties}
	\inf_{v_N\in V_N (\Omega)} \|u-v_N\|_{H^{\alpha/2} (\Omega)}  \leq C h^{s-\alpha/2} \|u\|_{H^{s} (\Omega)}. 
\end{eqnarray}
Under this assumption, we have the optimal error estimates for the Galerkin method.   
\begin{rem}
	Let  $ \mathcal{N}_{\Phi}(\Omega)$ be the native space for the RBF $\Phi$, see e.g., \cite{Wendland-2005,Narcowich-Ward-Wendland-2006}.  {\color{black}According to the reference \cite{Wendland-2005,Narcowich-Ward-Wendland-2006},  the native space is equivalent to the Sobolev space with the regularity index being infinity. For the test functions in Example 5.1 and 5.2  with infinitely smooth solutions such as Gaussian and GMQ RBF. }
	If the solution $u$ belongs to the native space, then we have spectral convergence:
	\begin{eqnarray}\label{rbf-approximation-properties-native-spaces}
		\inf_{v_N\in V_N (\Omega)} \|u-v_N\|_{H^{\alpha/2} (\Omega)}  \leq e^{-c/h}\|u\|_{\mathcal{N}_{\Phi}(\Omega)}. 
	\end{eqnarray}

 {\color{black}
 The approximation properties  of the generalized multi-quadratic functions in the assumption (3.17)  is a long-standing conjecture in the RBF community. Some progress has been made recently. For example, in the reference \cite{Ham-Ledford-2018}, the authors have proved the approximation properties when the index of GMQ  $\beta<-d-1/2$ or $\beta>1/2$ but $\beta$ is not an integer. Note that in this work we use $\beta=(\alpha-d)/2$.  When $d=1$, $\alpha\in(1,2]$  in the assumption \eqref{rbf-approximation-properties}  can be satisfied. However, for the other general cases, we are not able to verify this assumption.}
	
\end{rem}


\begin{thm}
	Let $u$ and $u_N$ solves \eqref{eqn:frac-Laplacian}-\eqref{eqn:frac-Laplacian-BC} and \eqref{galerkin-scheme} respectively. Assuming  \eqref{rbf-approximation-properties}, we have 
	\begin{eqnarray*}
		\|u-u_N\|_{H^{\alpha/2} (\Omega)} \leq C  h^{s-\alpha/2} \|u\|_{H^{s} (\Omega)}. 
	\end{eqnarray*}
	If the solution belongs to the native spaces, then we have the spectral convergence
	\begin{eqnarray*}
		\inf_{v_N\in V_N (\Omega)} \|u-v_N\|_{H^{\alpha/2} (\Omega)}  \leq e^{-c/h}\|u\|_{\mathcal{N}_{\Phi}(\Omega)}. 
	\end{eqnarray*}
\end{thm}

\begin{rem}
\color{black}    The index $s$ depends on the regularity of the solution. For example, for the problem with the zero truncated boundary condition,  the solution may be nonsmooth and $s=\alpha/2$. 
\end{rem}


\subsubsection{A non-conforming Galerkin  method for $\alpha \in (1,2)$}

As the solution to 
\eqref{eq:frac-poisson}
has compact support on $\Omega$, we adapt the global RBFs  below.

When $\alpha \in (1,2)$, however,  it is impossible to construct the approximated subspace which belongs to $\tilde{H}^{\alpha/2}(\Omega)=H_0^{\alpha/2}(\Omega) $. We can define the approximated solution space for practical computation by directly enforcing the zero boundary condition. More precisely, we  denote the trial space 
\begin{eqnarray}
	V_N (\bar{\Omega})= \mbox{Span} \{\phi_j(x), x_j\in X_{\bar{\Omega}}\}= {\color{black} \mbox{Span} \{L_j(x), x_j\in X_{\bar{\Omega}}\}},
\end{eqnarray}
and  define  the approximated solution space as
\begin{eqnarray}\label{nonconforming-subspace}
	V^o_{N} (\bar{\Omega})&:=&\mbox{Span}\{ {\color{black}L_j(x)}:  {\color{black}L_j(x)}\in V_{N}(\bar{\Omega})~\mbox{and}~ {\color{black}L_j(x_i)}=0~\mbox{for}~x_i\in X_{\partial\Omega} \}, \notag\\
	&=&\mbox{Span} \{L_j(x): {x_j\in X_{{\Omega}}}\}.
\end{eqnarray}  
Here $L_j(x)$'s are associated with all the nodal points $x_j\in \bar{\Omega}$, which are different from that in \eqref{conforming-subspace}. 

The nonconforming  Galerkin method for the problem \eqref{eqn:frac-Laplacian}-\eqref{eqn:frac-Laplacian-BC} is to find $u_N\in 		V^o_{N} (\bar{\Omega})$ such that 
\begin{eqnarray}\label{nonconform-galerkin-scheme}
	\tilde{a}(u_N,v_N)=((-\Delta)_{\Omega}^{\alpha/2}u_N,v_N) + (\rho_{\Omega}u_N,v_N) =(f,v_N), \quad \forall v_N\in 		V^o_{N} (\bar{\Omega}).
\end{eqnarray}

Similar to the case of $\alpha<1$,   we can readily show the well-posedness of discrete problem \eqref{nonconform-galerkin-scheme}  by  Lax-Milgram's theorem. In fact, taking the test function $v_N=u_N$, we have 
\begin{eqnarray*}
	\tilde{a}(u_N,u_N) = \frac{1}{2} c_{d,\alpha} \iint_{\Omega\times \Omega}  \frac{ |u_N(x)-u_N(y)|^2}
	{|x-y|^{d+\alpha}}\,dy dx +  \int_{\Omega} \rho_{\Omega} (x) |u_N(x)|^2 dx \geq c\|u_N\|^2_{ {H}^{\alpha/2} (\Omega)},
\end{eqnarray*}
where we have used the fact $ \rho_{\Omega} (x) \approx  \mbox{dist}(x,\partial \Omega)^{-\alpha}$ and  $\mbox{dist}(x,\partial \Omega)$ denotes the distance function. For the right-hand side, we have 
\begin{eqnarray*}
	(f,u_N)\leq \|f\|_{L^2 (\Omega)} \|u_N\|_{L^2(\Omega)} \leq c\|f\|_{L^2 (\Omega)} \|u_N\|_{H^{\alpha/2}(\Omega)} .
\end{eqnarray*}
Combining them together yields the stability estimate $\|u_N\|_{ {H}^{\alpha/2} (\Omega)} \leq c \|f\|_{L^2 (\Omega)} $. 

Next, we consider the convergence. Note that the solution to the continuous problem belongs to $\tilde{H}^{\alpha/2}(\Omega) = H^{\alpha/2}_0(\Omega)$  while the solution to the discrete problem belongs to $H^{\alpha/2}(\Omega)$.
Since the discrete space satisfies the zero boundary condition in one dimension, we can readily show the nonconforming method coincides with the conforming method so that we can obtain  the convergence straightforwardly. 
%
The discrete solution would converge to the continuous one if the constructed approximated solution sufficiently approximates the zero function on the boundary. Unfortunately, at this point, we cannot prove this conclusion in the multi-dimensional case when $d\geq 2$.   We will provide the numerical evidence in Section \ref{section-2d-numerical-example} to support our conclusion. 


\begin{rem}
	The above discussion does not cover the critical case $\alpha=1$ in 2D. The well-posedness of {\color{black}the discrete problem is similar}.  
\end{rem}

\section{An equivalent formulation and implementation}\label{section-collocation-reformulation-implementation}

In this section, we present a linear system from Galerkin formulations with numerical integration, equivalent to 
the collocation method
\eqref{eqn:discrete-interior}-\eqref{eqn:discrete-boundary} and use it in our 
simulations.  



For the Galerkin formulations, we need to discretize the integral in the continuous bilinear form \eqref{galerkin-scheme} or \eqref{nonconform-galerkin-scheme}. To this end, we introduce the numerical quadrature using the interpolants. Define 
\begin{eqnarray}
	Q_N(f):= \int_{\Omega}  f_N(x)dx=\sum_{x_j\in X_{\Omega}} f(x_j) w_{j},\quad w_j= \int_{\Omega} L_j(x) dx,
\end{eqnarray}
where the interpolant $ f_N(x)=\sum_{x_j\in X_{\Omega}} f(x_j) L_j(x) $. 
Let $u_N(x)=  \sum_{x_j\in \bar{\Omega}} u_N(x_j) L_j(x)$ be the numerical solution and denote  the identity operator by $ \mathbb{I}$. We are seeking the nodal value $u_N(x_j)$.
Then the Galerkin method with the numerical integration is to find 
\begin{eqnarray}\label{eq-scheme-num-int}
	\sum_{x_j\in X_{\Omega}} (	(-\Delta)_{\Omega}^{\alpha/2}+ \rho_{\Omega} \mathbb{I}) {u}_N (x)|_{x=x_j}v_N(x_j) w_{j}=\sum_{x_j\in X_{\Omega}} f(x_j) v_N(x_j) w_{j} ,\quad \forall v_N\in V_N (\Omega).
\end{eqnarray}

Assume that the integrand $f(x)$  belongs to the native spaces,  $f\in \mathcal{N}_{\Phi}(\Omega)$.  Under this assumption, we have  
\begin{eqnarray}
	\bigg|Q_N(f)-\int_{\Omega} f(x)dx\bigg| \leq e^{-c/h}\|f\|_{\mathcal{N}_{\Phi}(\Omega)},
\end{eqnarray}
From the estimate, we can see the effect of the numerical integration is negligible when the number of scatter points is large enough.

Plugging in the solution $u_N(x)=  \sum_{x_j\in \bar{\Omega}} u_N(x_j) L_j(x)$ into the equation \eqref{eq-scheme-num-int} and taking the test function as $v_N(x)= L_i(x)$,  we arrive at the linear system 
\begin{eqnarray}
	&& (	(-\Delta)_{\Omega}^{\alpha/2}+ \rho_{\Omega} \mathbb{I})  u_N(x_i)=  \sum_{x_j\in \bar{\Omega}} u_N(x_j)(	(-\Delta)_{\Omega}^{\alpha/2}+ \rho_{\Omega} \mathbb{I})   L_j(x_i)=f(x_i), \quad x_i\in \Omega,\\
	&&	u_N(x_i)=0,\quad x_i\in  \partial\Omega. 
\end{eqnarray}
or equivalently \begin{eqnarray}\label{rbf-linear-system}
	(	(-\Delta)_{\Omega}^{\alpha/2}+ \rho_{\Omega} \mathbb{I})  u_N(x_i)=  \sum_{x_j\in \Omega} u_N(x_j)(	(-\Delta)_{\Omega}^{\alpha/2}+ \rho_{\Omega} \mathbb{I})   L_j(x_i)=f(x_i), \quad x_i\in \Omega
\end{eqnarray}
due to the zero boundary conditions. 

Next, we need to deduce the coefficients  $(	(-\Delta)_{\Omega}^{\alpha/2}+ \rho_{\Omega} \mathbb{I})  L_j(x_i)$. 
The strategy is to transform the Lagrange functions basis $\{L_j(x)\}$ back into the shape functions basis $\{ \phi_j(x)\}$ and make use of the pseudo-spectral relation for the fractional Laplacian on the full space. To illustrate the idea, we
suppose that the number of scatter points for $X_{\Omega}$ and $X_{\bar{\Omega}}$ is $N^o$ and $N$, respectively. The difference $N^b=N-N^o$ denotes precisely the number of the scatter points on the boundary.
Introduce the vector-value functions,
\begin{eqnarray}
	{\bf{L}}(x)=[L_1(x),L_2(x),\cdots, L_{N}(x)]^T,\quad {\bf{\Phi}}(x)=[\phi_1(x),\phi_2(x),\cdots, \phi_{N}(x)]^T
\end{eqnarray}
and  the matrix  $A_{\phi}=[{\bf{\Phi}}(x_1),{\bf{\Phi}}(x_2),\cdots, {\bf{\Phi}}(x_N)]$.  Here $A_{\phi}$ is $N\times N$ symmetrical positive definite matrix due to the RBFs properties. 
Noting that the conversion between the Lagrange basis functions $\{L_j(x)\}$ and the shape functions  $\{\phi_j(x) \}$, that is $A_{\phi}{\bf{L}}(x)= {\bf{\Phi}}(x)$, we have 
\begin{eqnarray}\label{eq-matrix-vector-1}
	(	(-\Delta)_{\Omega}^{\alpha/2}+ \rho_{\Omega} \mathbb{I}) {\bf{L}}(x)= A^{-1}_{\phi} (	(-\Delta)_{\Omega}^{\alpha/2}+ \rho_{\Omega} \mathbb{I}) {\bf{\Phi}}(x),
\end{eqnarray}
where $j$-th component of the vector $ (	(-\Delta)_{\Omega}^{\alpha/2}+ \rho_{\Omega} \mathbb{I})  {\bf{\Phi}}(x)$ is $ (	(-\Delta)_{\Omega}^{\alpha/2}+ \rho_{\Omega} \mathbb{I}) \phi_j(x)$, which is easier to compute. In fact, 
using the pseudo-spectral relation in Theorem \eqref{RBF-Lemma-pseudo-spectral-relation},  we have 
\begin{eqnarray} 
	(	(-\Delta)_{\Omega}^{\alpha/2}+ \rho_{\Omega} \mathbb{I})	\phi_j(x)
	&=& (-\Delta)^{\alpha/2} \phi_j(x) + (-\Delta)^{\alpha/2}_{\Omega^c}\phi_j(x)   
	=	\mu_{d,\alpha}  \varepsilon^\alpha \psi_j(x)+ (-\Delta)^{\alpha/2}_{\Omega^c}\phi_j(x),
\end{eqnarray}
where $$
\psi_j(x) = ( \varepsilon^2+| x-x_j|^2)^{-\alpha/2-d/2}	,  \quad (-\Delta)^{\alpha/2}_{\Omega^c}\phi_j(x) := 
c_{d,\alpha}\int_{\Omega^c}\frac{\phi_j(y)} 
{|x-y|^{d+\alpha}} \, dy.$$ 
Denote ${\bf{\Psi}}(x)=[\psi_1(x),\psi_2(x),\cdots, \psi_{N}(x)]^T $.  Then  \eqref{eq-matrix-vector-1} is reduced to 
\begin{eqnarray}\label{eq-matrix-vector-2}
	(	(-\Delta)_{\Omega}^{\alpha/2}+ \rho_{\Omega} \mathbb{I}){\bf{L}}(x)= A^{-1}_{\phi} \bigg( {\bf{\Psi}}(x)  + (-\Delta)^{\alpha/2}_{\Omega^c}{\bf{\Phi}}(x) \bigg).
\end{eqnarray}
Recall that here we are only interested in the first $N^o$ components of $	(	(-\Delta)_{\Omega}^{\alpha/2}+ \rho_{\Omega} \mathbb{I}) {\bf{L}}(x)$. Thus we need to truncate the first $N^o$ rows of the inverse of  $N\times N$  matrix $A_{\phi}$. 

Denote $N\times N$ matrix  $A_{\psi}=[{\bf{\Psi}}(x_1),{\bf{\Psi}}(x_2),\cdots, {\bf{\Psi}}(x_N)]$ and $N\times N^o$ matrix  $$B^T_{\phi}=[ (-\Delta)^{\alpha/2}_{\Omega^c} {\bf{\Phi}}(x_1), (-\Delta)^{\alpha/2}_{\Omega^c} {\bf{\Phi}}(x_2),\cdots,  (-\Delta)^{\alpha/2}_{\Omega^c} {\bf{\Phi}}(x_{N^o})].$$
%
Then the linear system 
\eqref{rbf-linear-system} can be recast into the matrix-vector form
\begin{eqnarray}\label{eq-matrix-vector-3}
	\bigg(	(A_{\psi})_{1:N^o,1:N}+ B_{\phi} \bigg) (A_{\phi}^{-1})_{1:N,1:N^o} U=F ,
\end{eqnarray}
where $U=[u(x_1),u(x_2),\cdots u(x_{N^o})]^T $ and $F=[f(x_1),f(x_2),\cdots f(x_{N^o})]^T $.
In order to solve \eqref{eq-matrix-vector-3}, we introduce the intermediate variable $\Lambda=[\lambda_1,\lambda_2,\cdots, \lambda_N]^T $ and denote $N^b\times 1$ zero vector by ${\bf{0}} $. Then the system \eqref{eq-matrix-vector-3} is equivalent  to the following:
\begin{eqnarray}
	\bigg(	(A_{\psi})_{1:N^o,1:N}+ B_{\phi} \bigg) \Lambda &=&F,  \label{matrix-vector-eq-1} \\
	(A_{\phi})_{N^o+1:N,1:N}  \Lambda&=& {\bf{0}}, \label{matrix-vector-eq-11} 
\end{eqnarray}
and the evaluation step
\begin{eqnarray}
	U=	(A_{\phi})_{1:N^o,1:N} \Lambda      \label{matrix-vector-eq-2}.
\end{eqnarray}
It can be readily shown that this linear system is equivalent to 
\eqref{eqn:discrete-interior}-\eqref{eqn:discrete-boundary}.



\begin{rem}
	In our discretization, the resulting linear system consists of two generalized quadratic matrices. One of them is perturbated by the low-rank matrix. In work \cite{Burkard-Wu-Zhang-2021}, the dominant part of the linear matrix is the matrix with entries generated by the evaluation of hypergeometric functions at scatted data points. 
\end{rem}

\section{Numerical experiments}\label{section-Numerical-experiments}
In this section, we will study the performance of our collocation method for the equation with  smooth and nonsmooth data.
\begin{itemize}
	\item Case 1, the analytical solution  $u(x)=(1+|x|^2)^{-(d+1)/2}$, where $x\in \mathbb{R}^2$,  and the corresponding right-hand side is  $(-\Delta)^{\alpha/2}u(x)=\frac{\Gamma(d+\alpha)}{\Gamma(d)} \ _2F_1\bigg(\frac{d+\alpha}{2},\frac{d+\alpha+1}{2}; \frac{d}{2}, |x|^2 \bigg) $;
	see \cite[page 320]{Samko-2002}. 
	
	\item Case 2,  we test the compact support  function  $u (x)= (1-|x|^2)^p_+$ with $p>0$ on the bounded interval $(-1,1)$ in 1D and a unit disk in 2D, and the corresponding right-hand side (see \cite{Dyda12}) is   
	\begin{eqnarray}
		(-\Delta)^{\alpha/2}u(x) =  \frac{2^\alpha \Gamma(\frac{\alpha+d}{2}) \Gamma(p+1)}{\Gamma(d/2) \Gamma(-\alpha/2+p+1)} \, _2F_1 ((\alpha+d)/2, -p+\alpha/2, d/2, |x|^2 ),\quad x\in \mathbb{R}^2.
	\end{eqnarray}
\end{itemize}

For  simplicity, we use a direct solver (QR algorithm) to solve the problem and measure the error in the root mean square (RMS) sense as 
\begin{eqnarray}
	\mathcal{E}(N)=\bigg(\frac{1}{J} \sum_{j=1}^J\big[u(x_j)-u_N(x_j) \big]^2
	\bigg)^{1/2}\bigg/\bigg(\frac{1}{J} \sum_{j=1}^J |u(x_j) |^2
	\bigg)^{1/2},
\end{eqnarray}
and 
\begin{eqnarray}
	\widehat{	\mathcal{E}}(N)=\bigg(\frac{1}{J} \sum_{j=1}^J\big[(-\Delta)^{\alpha/2}u(x_j)-(-\Delta)^{\alpha/2}u_N(x_j) \big]^2
	\bigg)^{1/2}\bigg/\bigg(\frac{1}{J} \sum_{j=1}^J |(-\Delta)^{\alpha/2}u(x_j) |^2
	\bigg)^{1/2}.
\end{eqnarray}
We take a sufficiently large  $J$ so that the error $\mathcal{E}(N)$ is insensitive to the number of interpolation or test points.

For the integral on $\Omega^c$ described in Section \ref{subsuction-complement-domain}, we use Gauss–Jacobi quadrature 
rules with sufficiently large  $K,\,M$ in the numerical tests.

\subsection{One-dimensional examples}
By considering, solutions of different smoothness, we compare our results (GMQ) with  \cite{Burkard-Wu-Zhang-2021} using the Gaussian RBFs 
and fractional centered finite difference scheme (FDM) \cite{Hao-Zhang-Du-2021}. The Gaussian basis function we will examine is  \begin{equation}
	\phi_j(x)=	\exp(- |x-x_j|^2/\varepsilon^2), \label{def-shape-parameter-g}
\end{equation}
where $\varepsilon$ is called a shape parameter. 
	For the meshfree method using Gaussian RBFs in \cite{Burkard-Wu-Zhang-2021}, one  needs to compute the tail integral
	\begin{eqnarray*}
		(-\Delta)^{\alpha/2}_{\Omega^c}\phi_j(x_i) = 
		c_{d,\alpha}\int_{\Omega^c}\frac{\phi_j(y)} 
		{|x_i-y|^{d+\alpha}} \, dy= 	c_{d,\alpha}\int_{|y|\geq 1}\frac{\exp(- |y-x_j|^2/\varepsilon^2)} 
		{|x_i-y|^{d+\alpha}} \, dy.
	\end{eqnarray*}
	Due to the exponential decay of the integrand, the authors in \cite{Burkard-Wu-Zhang-2021} suggested truncating the unbounded domain to compute the integral. In the following tests, we take the computational domain as $1\leq |y|\leq 5$ (We have also tried the larger domain size and found almost no accuracy improvement).   

First, we consider the piece-wise analytical function  $(1-|x|^2)_+$ and present the numerical results in Table \ref{table-comparasion-1d-piecewise-smooth-1d}. 
Compared to FDM, the GMQ method is efficient in using less number of the degrees of freedom to achieve relatively high accuracy. In contrast, FDM yields less accurate results even with a large number of grid points. In particular, when the order $\alpha$ gets large, the accuracy deteriorates for FDM.  {\color{black} In particular, the finite difference method does not converge for alpha = 1.6.  The reason is that, for the finite difference scheme to converge, it requires the solution to have no less than
    $\alpha$-th derivative while the 
  the regularity index of the test function $(1-x^2)_+$ is  no more than 1.5 which is below the differential order $\alpha=1.6$.  } 
Also, the  GMQ method even achieves the accuracy of order $10^{-7}$ for  $\alpha=1.8$. 
Compared with the Gaussian RBFs method in \cite{Burkard-Wu-Zhang-2021}, our method is  more accurate, especially when 
$\alpha$ is large. 

\begin{table}[!ht]
	\centering
	\caption{ Comparison  among  generalized multiquadric (GMQ), Gaussian RBFs  methods and the finite difference method (FDM) in \cite{Hao-Zhang-Du-2021} for the exact solution $u(x)=(1-|x|^2)_+$. 
		The shape parameters in \eqref{def-shape-parameter-g} and \eqref{def-shape-parameter-gmq} are $1$ for Gaussian and $1.5$ for GMQ,  respectively. Test points are taken the same as the finite difference method.  }\label{table-comparasion-1d-piecewise-smooth-1d}
	\begin{tabular}{cccccccccc}  
		\hline \hline
		&	& \multicolumn{2}{c}{ $\alpha=0.4$}&   \multicolumn{2}{c}{ $\alpha=0.8$}  &\multicolumn{2}{c}{ $\alpha=1.2$}&\multicolumn{2}{c}{ $\alpha=1.6$} \\
		\cline{3-4} \cline{5-6} \cline{7-8}\cline{9-10}
		&$N$ &  $ \widehat{\mathcal{E}}(N)$ & rate & $\widehat{\mathcal{E}}(N)$ & rate  & $\widehat{\mathcal{E}}(N)$ & rate & $\widehat{\mathcal{E}}(N)$ & rate  \\
		\hline
		GMQ  & 2       & 5.64e-02     &          & 8.48e-02     &          & 1.22e-01     &          & 1.59e-01     &          \\ 
		&4       & 1.00e-02     &     2.50     & 1.74e-02     &     2.29     & 2.64e-02     &     2.21     & 3.70e-02     &     2.10    \\ 
		&8       & 1.11e-04     &     6.50     & 2.31e-04     &     6.23     & 4.56e-04     &     5.86     & 8.08e-04     &     5.52    \\ 
		&16       & 2.06e-06     &     5.75     & 1.65e-06     &     7.12     & 8.18e-07     &     9.12     & 4.06e-07     &    10.96    \\

		\hline
		Gaussian  & 2       & 1.85e-01     &          & 2.75e-01     &          & 4.02e-01     &          & 5.88e-01     &          \\ 
		&4       & 3.05e-02     &     2.60     & 5.59e-02     &     2.30     & 9.97e-02     &     2.01     & 1.84e-01     &     1.68    \\ 
		&8       & 1.17e-03     &     4.70     & 5.98e-03     &     3.22     & 3.02e-02     &     1.72     & 1.37e-01     &     0.42    \\ 
		&16       & 9.97e-04     &     0.23     & 7.02e-03     &    -0.23     & 5.47e-02     &    -0.86     & 2.52e-01     &    -0.87    \\ 
		\hline
		
		FDM &   128       & 3.14e-04     &          & 1.67e-03     &          & 5.14e-03     &          & 7.82e-03     &          \\ 
		&	256       & 1.48e-04     &     1.09     & 1.03e-03     &     0.70     & 4.10e-03     &     0.32     & 7.83e-03     &    -0.00    \\ 
		&	512       & 6.91e-05     &     1.09     & 6.35e-04     &     0.70     & 3.28e-03     &     0.32     & 7.70e-03     &     0.02    \\ 
		&	1024       & 3.23e-05     &     1.10     & 3.91e-04     &     0.70     & 2.62e-03     &     0.32     & 7.53e-03     &     0.03    \\   
		\hline \hline
	\end{tabular}
\end{table}

We then test a globally smoother function  $(1-|x|^2)^2_+$. Table \ref{table-comparasion-1d-piecewise-smooth-2-1d} shows that FDM and Gaussian methods have a substantial increase in the accuracy compared to results for the previous globally less regular function. But the accuracy of these two methods is still lower than that of our method. When $\alpha=1.6$, our method has $10^{-6}$ accuracy while the accuracy for Gaussian RBFs is of order $10^{-3}$

\begin{table}[!ht]
	\centering
	\caption{  Comparison  among  generalized multiquadric (GMQ), Gaussian RBFs  methods and the finite difference method (FDM)  in \cite{Hao-Zhang-Du-2021}  for the piece-wise function $u(x)=(1-|x|^2)^2_+$. The shape parameters are $1$ for Gaussian and $1.5$ for GMQ,  respectively. Test points are taken the same as the finite difference method. }\label{table-comparasion-1d-piecewise-smooth-2-1d}
	\begin{tabular}{cccccccccc}  
		\hline \hline
		&	& \multicolumn{2}{c}{ $\alpha=0.4$}&   \multicolumn{2}{c}{ $\alpha=0.8$}  &\multicolumn{2}{c}{ $\alpha=1.2$}&\multicolumn{2}{c}{ $\alpha=1.6$} \\
		\cline{3-4} \cline{5-6} \cline{7-8}\cline{9-10}
		&$N$ &  $\widehat{\mathcal{E}}(N)$ & rate & $\widehat{\mathcal{E}}(N)$ & rate  & $\widehat{\mathcal{E}}(N)$ & rate & $\widehat{\mathcal{E}}(N)$ & rate  \\
		\hline
		GMQ	& 2       & 5.37e-02     &          & 9.28e-02     &          & 1.50e-01     &          & 2.22e-01     &          \\ 
		&4       & 2.07e-02     &     1.37     & 3.81e-02     &     1.28     & 6.19e-02     &     1.28     & 9.51e-02     &     1.22    \\ 
		& 8       & 3.68e-04     &     5.82     & 7.99e-04     &     5.58     & 1.69e-03     &     5.20     & 3.28e-03     &     4.86    \\ 
		& 16       & 8.74e-06     &     5.40     & 7.35e-06     &     6.76     & 3.96e-06     &     8.74     & 2.38e-06     &    10.43    \\ 
		
		\hline
		Gaussian	&2       & 1.26e-01     &          & 1.76e-01     &          & 2.47e-01     &          & 3.42e-01     &          \\ 
		&	4       & 5.51e-02     &     1.19     & 9.21e-02     &     0.94     & 1.50e-01     &     0.72     & 2.39e-01     &     0.52    \\ 
		&	8       & 1.45e-03     &     5.25     & 3.14e-03     &     4.88     & 6.74e-03     &     4.48     & 1.56e-02     &     3.93    \\ 
		&	16       & 3.13e-05     &     5.53     & 2.04e-04     &     3.95     & 9.10e-04     &     2.89     & 2.11e-03     &     2.89    \\ 
		\hline
		FDM		&    128       & 4.04e-05     &          & 1.53e-04     &          & 4.70e-04     &          & 1.09e-03     &          \\ 
		&	256       & 1.06e-05     &     1.93     & 4.61e-05     &     1.73     & 1.82e-04     &     1.37     & 5.64e-04     &     0.95    \\ 
		&	512       & 2.75e-06     &     1.95     & 1.38e-05     &     1.74     & 7.16e-05     &     1.35     & 2.96e-04     &     0.93    \\ 
		&	1024       & 7.07e-07     &     1.96     & 4.14e-06     &     1.74     & 2.84e-05     &     1.33     & 1.57e-04     &     0.92    \\ 
		
		\hline \hline
	\end{tabular}
\end{table}

Last, we test the function  $(1-|cx|^2)^{\alpha}_+$, which has a weak singularity at the boundary when $c=1$ and in the interior when $c=2$. 
From Table \ref{table-comaparison-singular-function-1d}, we see that all methods have low accuracy and achieve almost the same performance. 
For the RBFs methods, the accuracy deteriorates for a relatively large  $N$.
For the solution with only interior weak singularity $c=2$, we observe stable performance as demonstrated in Table \ref{table-interior-singularity-1d}.

Here  we apply the scaling property \eqref{scaling-property} to obtain,  when
$v(x)=(1-|x|^2)^{p}_+$,  $$	(-\Delta)^{\alpha/2}v(lx)= l^{\alpha}	\frac{2^\alpha \Gamma(\frac{\alpha+d}{2}) \Gamma(p+1)}{\Gamma(d/2) \Gamma(-\alpha/2+p+1)} \, _2F_1 ((\alpha+d)/2, -p+\alpha/2, d/2, |l x|^2 ) .$$

\begin{table}[!ht]
	\centering
	\caption{  Comparison  among  generalized multiquadric (GMQ), Gaussian RBFs  methods and the finite difference method (FDM)  in \cite{Hao-Zhang-Du-2021}  for the weakly singular  function $u(x)=(1-|x|^2)^{\alpha}_+$.
		The shape parameters for both RBFs methods are $\varepsilon=2h$ with $h=2/N$.  }\label{table-comaparison-singular-function-1d}
	\begin{tabular}{cccccccccc}  
		\hline \hline
		&	& \multicolumn{2}{c}{ $\alpha=0.4$}&   \multicolumn{2}{c}{ $\alpha=0.8$}  &\multicolumn{2}{c}{ $\alpha=1.2$}&\multicolumn{2}{c}{ $\alpha=1.6$} \\
		\cline{3-4} \cline{5-6} \cline{7-8}\cline{9-10}
		&$N$ &  $\widehat{\mathcal{E}}(N)$ & rate & $\widehat{\mathcal{E}}(N)$ & rate  & $\widehat{\mathcal{E}}(N)$ & rate & $\widehat{\mathcal{E}}(N)$ & rate  \\
		\hline
		GMQ	&   64       & 1.73e-02     &          & 1.85e-02     &          & 5.75e-03     &          & 1.20e-02     &          \\ 
		&128       & 9.19e-03     &     0.91     & 1.02e-02     &     0.86     & 8.24e-03     &    -0.52     & 5.79e-03     &     1.05    \\ 
		&256       & 6.68e-03     &     0.46     & 1.32e-02     &    -0.38     & 2.26e-02     &    -1.46     & 2.83e-02     &    -2.29    \\ 
		&512       & 1.28e-02     &    -0.94     & 3.44e-02     &    -1.38     & 5.75e-02     &    -1.35     & 5.69e-02     &    -1.01    \\ 
		\hline
		Gaussian	&	64       & 1.50e-02     &          & 1.25e-02     &          & 5.11e-03     &          & 2.12e-02     &          \\ 
		& 128       & 5.20e-03     &     1.53     & 3.61e-03     &     1.79     & 7.61e-03     &    -0.57     & 1.47e-02     &     0.53    \\ 
		& 256       & 9.19e-03     &    -0.82     & 6.92e-03     &    -0.94     & 2.67e-03     &     1.51     & 8.72e-03     &     0.75    \\ 
		&	512       & 1.57e-02     &    -0.78     & 1.65e-02     &    -1.25     & 5.26e-03     &    -0.98     & 5.63e-03     &     0.63    \\ 
		\hline
		FDM		&    64       & 1.47e-02     &          & 1.10e-02     &          & 2.27e-03     &          & 7.90e-03     &          \\ 
		&	128       & 1.04e-02     &     0.50     & 7.72e-03     &     0.51     & 1.47e-03     &     0.63     & 5.24e-03     &     0.59    \\ 
		& 256       & 7.34e-03     &     0.50     & 5.43e-03     &     0.51     & 9.86e-04     &     0.57     & 3.59e-03     &     0.55    \\ 
		&	512       & 5.19e-03     &     0.50     & 3.82e-03     &     0.51     & 6.78e-04     &     0.54     & 2.49e-03     &     0.53    \\ 
		\hline \hline
	\end{tabular}
\end{table}

\begin{table}[!htb]
	\centering
	\caption{Generalized multiquadric RBFs method  for the weakly singular function   $u(x)=(1-|2x|^2)^{\alpha}_+$. The shape parameter $\varepsilon = 2h$ with $h=2/N$.  }\label{table-interior-singularity-1d}
	\begin{tabular}{ccccccccc}  
		\hline \hline
		& \multicolumn{2}{c}{ $\alpha=0.4$}&   \multicolumn{2}{c}{ $\alpha=0.8$}  &\multicolumn{2}{c}{ $\alpha=1.2$}&\multicolumn{2}{c}{ $\alpha=1.6$} \\
		\cline{2-3} \cline{4-5} \cline{6-7}\cline{8-9}
		$N$ &  $\widehat{\mathcal{E}}(N)$ & rate & $\widehat{\mathcal{E}}(N)$ & rate  & $\widehat{\mathcal{E}}(N)$ & rate & $\widehat{\mathcal{E}}(N)$ & rate  \\
		\hline
		256       & 1.61e-02     &          & 2.15e-02     &          & 1.20e-02     &          & 5.59e-03     &          \\ 
		512       & 1.14e-02     &     0.50     & 1.51e-02     &     0.51     & 8.33e-03     &     0.53     & 3.92e-03     &     0.51    \\ 
		1024       & 8.06e-03     &     0.50     & 1.06e-02     &     0.51     & 5.81e-03     &     0.52     & 2.76e-03     &     0.51    \\ 
		2048       & 5.70e-03     &     0.50     & 7.49e-03     &     0.50     & 4.07e-03     &     0.51     & 1.94e-03     &     0.51    \\ 
		
		\hline \hline
	\end{tabular}
\end{table}


\subsection{Two-dimensional steady problems}
\label{section-2d-numerical-example}


In the 2D  case, we  {\color{black} first}  consider a unit disk domain. 
For a fair comparison with Gaussian RBFs, we use the same setup as in \cite{Burkard-Wu-Zhang-2021}, i.e.,
\begin{eqnarray}\label{gridpoints-layout}
	X=	X_{\Omega} \cap X_{\partial \Omega} = \biggl\{\frac{l}{L} \bigg(\cos(2j\pi/(J+1)), \sin(2j\pi/(J+1))\bigg)  ~\mbox{for}~ 0\leq l\leq L,0\leq j\leq J  \biggr\}
\end{eqnarray}  
with the total number of data points $N=L*(J+1)+1$. In the numerical tests, we take $L=J=3,5,7,9,11$, and  $N=13,31,57,91,133$, correspondingly. From Table 
\ref{table:2D-smooth-case}, we can see a fast/spectral convergence of GMQ method, for the smooth solution, as in  the Gaussian RBF in \cite{Burkard-Wu-Zhang-2021}. 

\begin{table}[!htb] 
	\centering
	{	
		\caption{ The root mean square error  $\mathcal{E}(N)$ and the condition number $\mathcal{K}(N)$ in 2D when the solution is $u(x)=(1+|x|^2)^{-3/2}$. }\label{table:2D-smooth-case}
		\begin{tabular}{ccccccc}
			\hline\hline
			& \multicolumn{2}{c}{ $\varepsilon=1$} & \multicolumn{2}{c}{ $\varepsilon=1.5$} & \multicolumn{2}{c}{ $\varepsilon=2$} \\
			\cline{2-3} \cline{4-5} \cline{6-7}
			$N$ &  $\mathcal{E}(N)$ & $\mathcal{K}(N)$ &  $\mathcal{E}(N)$ & $\mathcal{K}(N)$&  $\mathcal{E}(N)$ & $\mathcal{K}(N)$  \\
			\hline
			13       & 1.86e-02     & 1.96e+03    & 2.50e-02     & 5.41e+04    & 7.02e-02     & 8.58e+05   \\  
			31       & 1.97e-03     & 3.80e+06    & 9.72e-04     & 2.43e+09    & 3.71e-03     & 4.79e+11   \\  
			57       & 2.27e-04     & 5.39e+10    & 2.92e-05     & 9.18e+14    & 3.77e-04     & 4.28e+16   \\  
			91       & 2.82e-05     & 2.17e+15    & 1.45e-06     & 5.22e+17    & 5.09e-04     & 2.31e+18   \\  
			133       & 2.98e-06     & 5.31e+17    & 1.35e-07     & 2.37e+18    & 3.30e-05     & 4.19e+19   \\  
			\hline \hline  
		\end{tabular}
} 	\end{table}

We also test a less smooth solution. Consider the equation with the nonsmooth solution $ (1-|x|^2)^{1+\alpha/2} $, we demonstrate the accuracy  in Table  \ref{table:2D_gauss_vs_multiq} and CPU time in  Table \ref{table:cpu_time_comp}. From Table  \ref{table:2D_gauss_vs_multiq}, GMQ method can achieve almost the same performance as Gaussian RBFs, using the same degrees of freedom of grid points. However, the computational cost of the Gaussian RBF method is several orders of magnitude larger than our  GMQ RBF method.

\begin{table}[!htb]
	\centering
	\caption{ Comparison  between  generalized multiquadric (GMQ) and Gaussian RBFs  method when the solution is function $u(x)=(1-|x|^2)^{1+\alpha/2}_+$.  The shape parameters are taken as $\varepsilon= 2h $ for both GMQ and Gaussian RBFs.  }\label{table:2D_gauss_vs_multiq}
	\begin{tabular}{cccccccccc}  
		\hline \hline
		&	& \multicolumn{2}{c}{ $\alpha=0.4$}&   \multicolumn{2}{c}{ $\alpha=0.8$}  &\multicolumn{2}{c}{ $\alpha=1.2$}&\multicolumn{2}{c}{ $\alpha=1.6$} \\
		\cline{3-4} \cline{5-6} \cline{7-8}\cline{9-10}
		&$N$ &  $\mathcal{E}(N)$ & rate & $\mathcal{E}(N)$ & rate  & $\mathcal{E}(N)$ & rate & $\mathcal{E}(N)$ & rate  \\
		\hline
		GMQ  & 13       & 4.92e-02     &          & 5.23e-02     &          & 6.39e-02     &          & 7.21e-02     &          \\ 
		&	53       & 1.88e-02     &     1.39     & 1.44e-02     &     1.86     & 1.32e-02     &     2.27     & 1.85e-02     &     1.96    \\ 
		&	209       & 5.80e-03     &     1.69     & 5.78e-03     &     1.32     & 5.91e-03     &     1.16     & 7.09e-03     &     1.38    \\ 
		&	825       & 1.37e-03     &     2.08     & 1.33e-03     &     2.11     & 1.26e-03     &     2.23     & 1.16e-03     &     2.61    \\ 
		&	3269       & 7.51e-04     &     0.87     & 7.58e-04     &     0.82     & 5.12e-04     &     1.30     & 4.97e-04     &     1.22    \\

		\hline		
		Gaussian & 13       & 1.66e-02     &          & 2.91e-02     &          & 4.50e-02     &          & 7.78e-02     &          \\ 
		& 53       & 6.14e-03     &     1.43     & 8.33e-03     &     1.80     & 1.65e-02     &     1.45     & 4.00e-02     &     0.96    \\ 
		& 209       & 4.08e-03     &     0.59     & 7.36e-03     &     0.18     & 9.13e-03     &     0.85     & 1.19e-02     &     1.76    \\ 
		& 825        & 2.32e-03     &     0.82     & 2.16e-03     &     1.77     & 1.58e-03     &     2.53     & 1.37e-03     &     3.12    \\

		\hline \hline
	\end{tabular}
\end{table}

\begin{table}[!htb]
	\centering
	\caption{   Comparison of the CPU time (measured in seconds) between generalized multiquadric (GMQ)  and Gaussian (G) RBFs method (The shape parameters for both methods are $\varepsilon = 2h$ with $h=2/N$)    for the nonsmooth solution $u(x)=(1-|x|^2)^{1+\alpha/2}_+$. }\label{table:cpu_time_comp}
	\begin{tabular}{ccccccccc}  
		\hline \hline
		& \multicolumn{2}{c}{ $\alpha=0.4$}&   \multicolumn{2}{c}{ $\alpha=0.8$}  &\multicolumn{2}{c}{ $\alpha=1.2$}&\multicolumn{2}{c}{ $\alpha=1.6$} \\
		\cline{2-3} \cline{4-5} \cline{6-7}\cline{8-9}
		$N$ &  GMQ & G & GMQ & G  & GMQ & G & GMQ & G  \\
		\hline
		13       & 6.48e-02     &   8.55e-01       & 1.33e-02     &   7.58e-02       & 1.30e-02     & 4.74e-02         & 1.41e-02     & 4.42e-02         \\ 
		53       & 4.94e-02     &   4.39e-01       & 3.63e-02     &  4.70e-01        & 4.50e-02     &   5.34e-01       & 3.57e-02     & 6.63e-01         \\ 
		209       & 1.37e-01     &    1.32e+01      & 1.29e-01     &   1.29e+01       & 1.33e-01     &  1.61e+01        & 1.25e-01     &  1.33e+01        \\ 
		825       & 7.48e-01     &    1.77e+02      & 5.86e-01     &      1.85e+02    & 5.80e-01     &    1.94e+02      & 5.74e-01     &   1.90e+02       \\ 
		3269       & 9.38e+00     &  $\geq$ 6 hours        & 9.30e+00     &  $\geq$ 6 hours        & 9.29e+00     &  $\geq$ 6 hours        & 9.18e+00     &   $\geq$ 6 hours       \\

		\hline \hline
	\end{tabular}
\end{table}

In the next example, we consider the elliptic equation with smooth right-hand side data $f=1$. The explicit solution is  $(1-|x|^2)_+^{\alpha/2}/\lambda_0^\alpha$ with $\lambda_0^\alpha=\Gamma(1+\alpha/2)^2$.  In the numerical test, we take  $N=111$ and $\varepsilon=0.8$ to compute the numerical solution and provide the comparison of the solution profiles for different fractional-order $\alpha$ in Figure  \ref{numerical-solution-2D} and the point-wise error in Figure \ref{pointwise-error-numerical-solution-2D}. 
From Figure  \ref{pointwise-error-numerical-solution-2D}, we can see the error near the boundary dominates, and the error gets smaller when $\alpha$ becomes larger, which agrees with the regularity  of the solution. 

{\color{black}
We also consider  a square domain $[-\sqrt{2}/2,\sqrt{2}/2 ]^2$ . For a non-disk domain, we use the strategy described in the Remark \ref{rem-reformulation-of-the-scheme} and embed the smaller square domain into the unit disk. We take the uniform grid points over the domain $[-1,1]^2$ with the step-size $h=1/32$. The shape parameter is taken as $\varepsilon=0.05$. The solution profile is shown in Figure \ref{numerical-solution-square-2D}, where we observe that the solution to the homogeneous boundary value problem with the smooth right-hand side has the sharp edge near the boundary for the smaller value of $\alpha$. This is consistent with the asymptotic regularity estimates in the literature. }

\begin{figure}[!ht]
	\centering
	\includegraphics[width=0.8\linewidth]{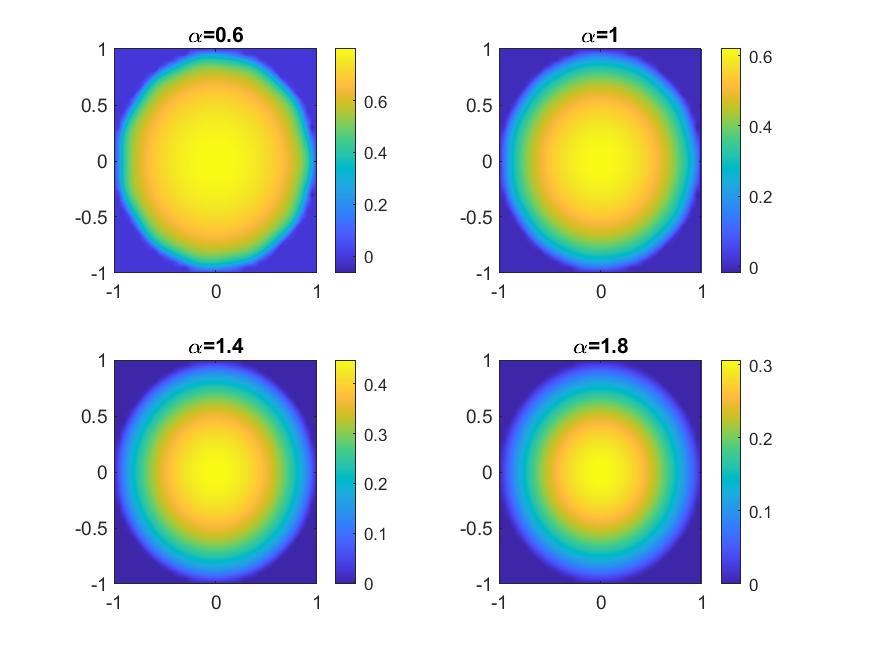}
	\caption{Numerical solution of the fractional Poisson problem cross reference on a unit disk domain for $f=1$ by GMQ with the shape parameter $\varepsilon=1$. }\label{numerical-solution-2D}
\end{figure}

\begin{figure}[!ht]
	\centering
	\includegraphics[width=0.8\linewidth]{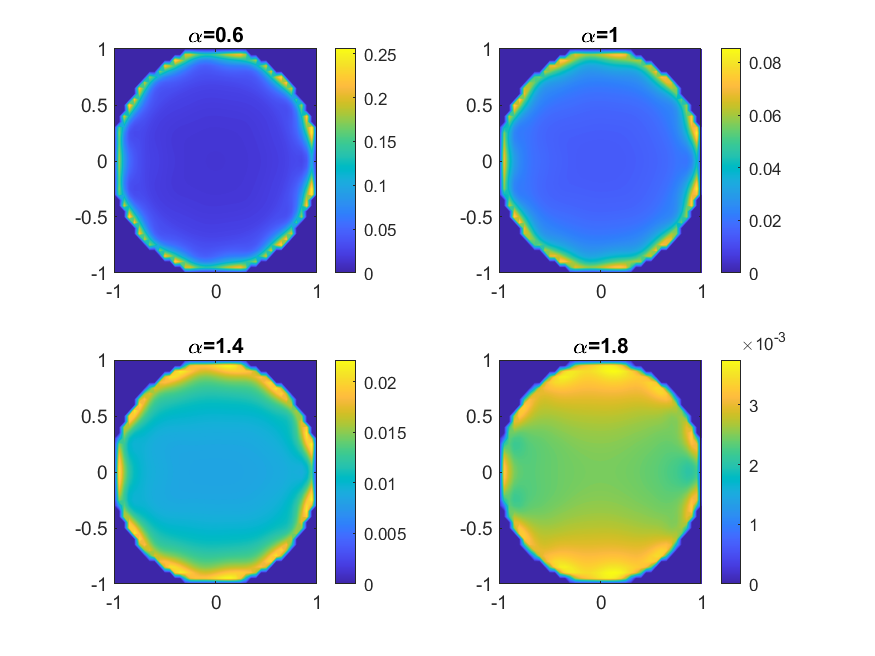}
	\caption{Pointwise  errors of GMQ with the shape parameter $\varepsilon=1$ solving the fractional Poisson problem cross reference on a unit disk. }\label{pointwise-error-numerical-solution-2D}
\end{figure}

\begin{figure}[!ht]
	\centering
	\includegraphics[width=0.8\linewidth]{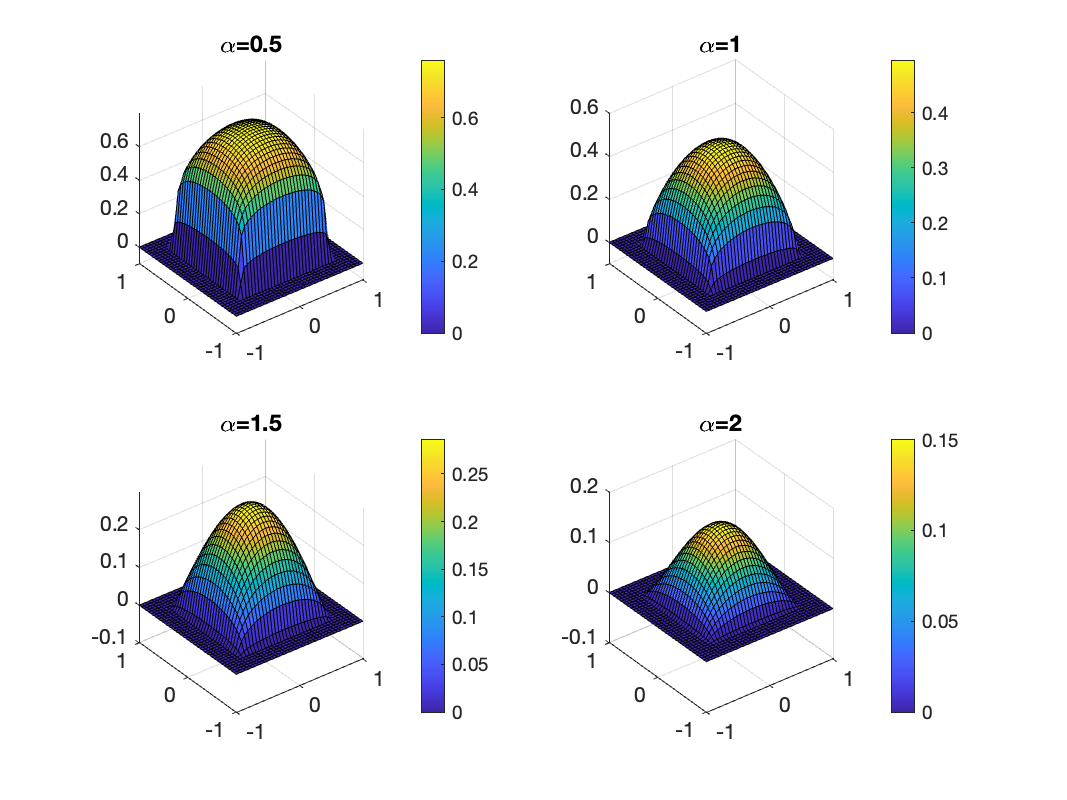}
	\caption{{\color{black}Numerical solution of the fractional Poisson problem cross reference on a square domain for $f=1$ by GMQ with the shape parameter $\varepsilon=0.05$ and the uniform grid points step-size $h=1/32$.} }\label{numerical-solution-square-2D}
\end{figure}

\subsection{Two dimensional time-dependent problems}
Consider the following parabolic  equation with fractional Laplacian 
\begin{eqnarray}\label{mixed-diffusion}
	u_t+ \chi(-\Delta)^{\alpha/2}u +(1-\chi) (-\Delta)u=0,\, x\in \set{x\in \mathbb{R}^2\,|\,|x|\leq 1}
\end{eqnarray}
subjected to the smooth initial data
$	u(x,0)=	e^{-16x_1^2-4x_2^2}$. 
This model is motivated by applications where
the two transport mechanisms are typically present: $-\Delta u$ for classical 
diffusion  and    $(-\Delta)^{\alpha/2}u$ for anomalous diffusion; see \cite{del-Castillo-Negrete-2014,Epps18},

First, we investigate the effect of the nonlocal term on diffusion. We use the standard second-order Crank-Nicolson scheme in time with the step size $dt=0.001$.
For the spatial discretization, we use the GMQ RBF collocation methods with the grid points in  \eqref{gridpoints-layout} with $L=J=8$. The total number of points is $N=73$. The snapshots of the solution for the model problem with different parameters $\theta$ are shown in  {\color{black}Figure \ref{comparison-mixed-diffusion}.}
We observe from the figure that the fractional Laplacian term in the model ($\chi=1$ or $\chi=0.5$) slows down the diffusion ($\chi=0$).

\begin{figure}[!ht]
	\centering
	\includegraphics[width=1\linewidth]{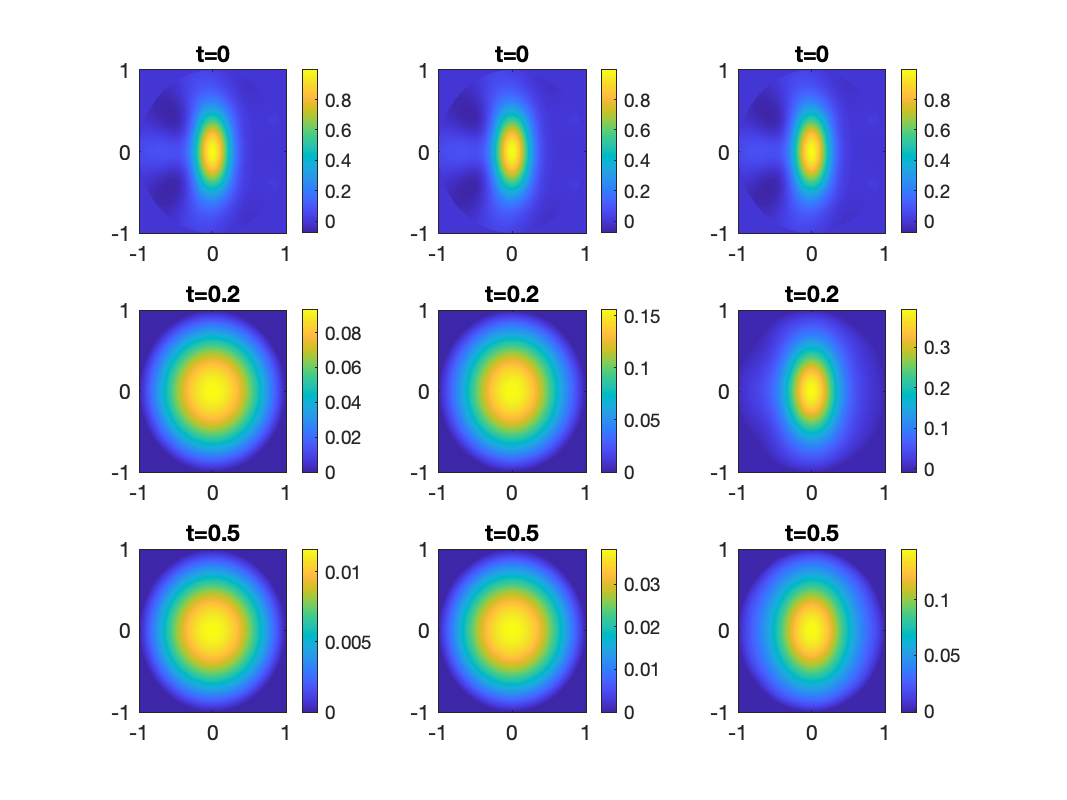}
	\caption{Comparison of the dynamics between the  standard diffusion ($\chi=0$, Left), the mixed diffusion ($\chi=0.5$, Middle) and fractional diffusion  ($\chi=1$, Right) for the model problem \eqref{mixed-diffusion}
		with $\alpha=1$. }\label{comparison-mixed-diffusion}
\end{figure}

In the next example, we consider the application to the quasi-geostrophic
flow,
\begin{eqnarray}
	&& \partial_t \theta + \mathbf{u}\cdot \nabla \theta + \kappa (-\Delta)^{\alpha/2} \theta =0, \notag\\
	&& \mathbf{u}=\nabla^{\perp} \Psi, \notag\\
	&& (-\Delta)^{1/2} \Psi = -\theta,  \label{eq-quasi-geotrophic}
\end{eqnarray}
where $\nabla^{\perp} \Psi= ( -\partial _{x_2} \Psi, \partial_{x_1} \Psi) $.  
In numerical simulations, we consider a single vortex, resembling cyclonic circulations within the atmosphere. We use the explicit strong stability preserving Runge–Kutta method with 3 stages and of order 3 (the Shu-Osher scheme) with time stepsize $dt=0.01$. {\color{black}In space discretization, we take the shape parameter $\varepsilon=0.1$ in the computation  and uniform grid points with the stepsize $h=1/16$ inside the unit disk. } 
We consider the initial condition $\theta(x,0)= e^{-4x_1^2-64x_2^2}$ and present the time evolution of the field $\theta$ is represented in {\color{black}Figure  \ref{isotrophic-flow}}, at different time instants.
We observe that the small eccentricity imposed by the initial condition results in the lack of filaments or secondary vortices and thus the single vortex tends to collapse into a circular configuration.

\begin{figure}[!ht]
	\centering
	\includegraphics[width=1\linewidth]{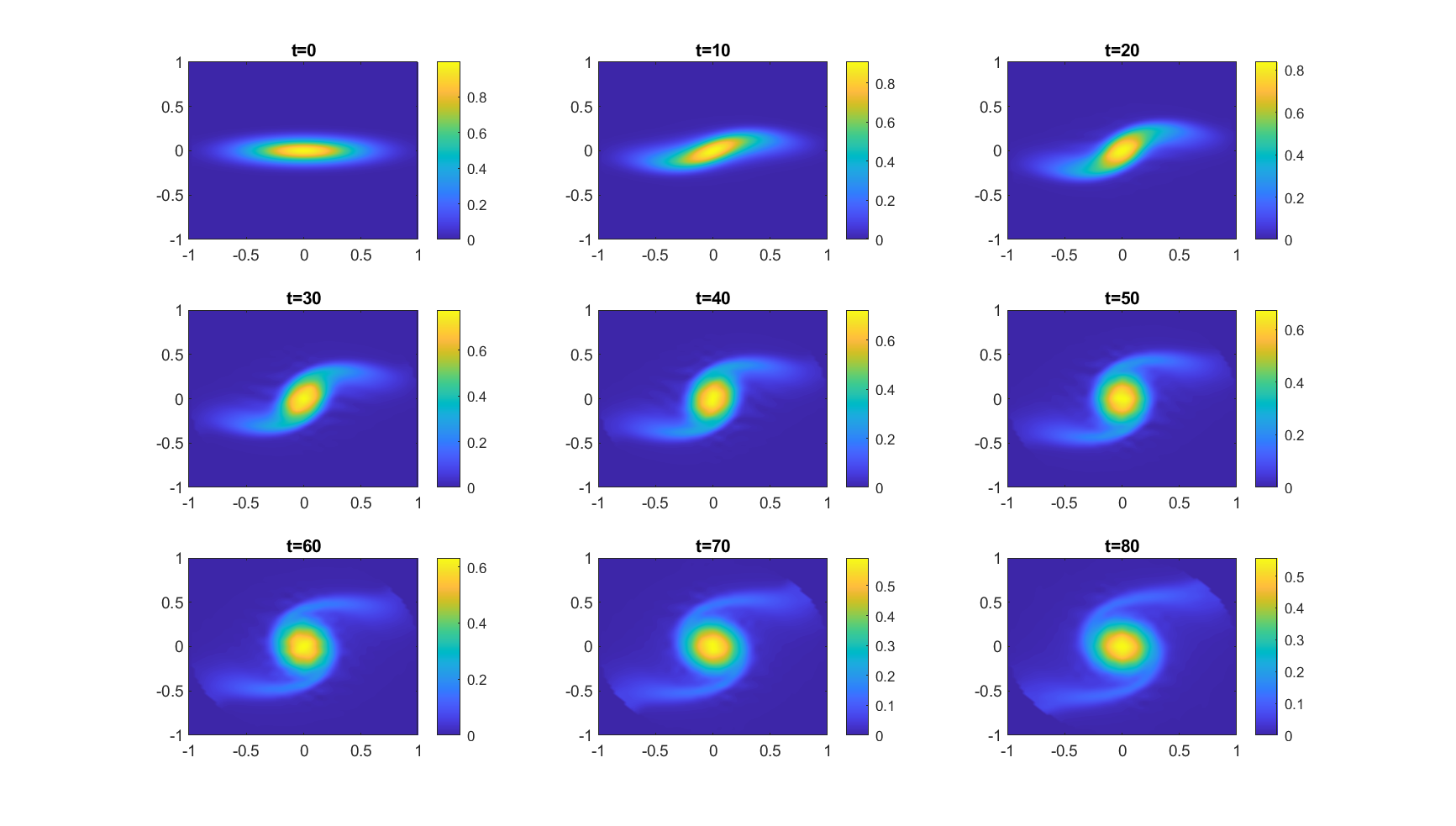}
	\caption{The time evolution of the field $\theta$    for the model problem \eqref{eq-quasi-geotrophic} with $\kappa=0.001$. }\label{isotrophic-flow}
\end{figure}

\section{Summary and Discussions}
We present a new meshless method based on the generalized multi-quadratic RBFs for the integral fractional Laplacian. 
With the power of the generalized multi-quadratic RBF being the order of   fractional Laplacian minus half of the dimension, we discover that the integral fractional Laplacian of it is still a generalized multi-quadratic RBF.
We then utilize this discovery and design a collocation method for equations with fractional Laplacian. 
By establishing the equivalence between the collocation method and a Galerkin method, we obtain the convergence of the collocation method. 
We also provide a simple implementation. 
%
Finally, we demonstrate with numerical examples the accuracy and efficiency of our method compared to the existing method using the Gaussian RBF.

Some directions are worth further exploration. First, the solution to fractional boundary value problems may admit the weaker singularity (first- or second-order derivatives do not exist) near the boundary. 
Condition numbers of the linear system can be large and thus the accuracy of the method can degenerate.
As the condition number 
is closely related to the smoothness of the RBFs and their shape parameters,
some techniques in  \cite{Sarra-Kansa-2009} may help to alleviate the ill-conditioning issues.  {\color{black} Based on the numerical experiments, we give a guidance to choose the shape parameter.  In general, if the solution is smooth, we choose a large parameter for out method. If the solution is non-smooth, the parameter is chosen proportional to the step-size to alleviate the ill-condition issue. However, how to pick the best shape parameters is important but nontrivial task. One of the approaches is to consider the adaptive method which is beyond the scope of this work.}
Second, our method can be combined with existing fast solvers in the RBFs literature.  {\color{black}We hope to report these relevant results for the unsolved issues discussed in this work in our future research.}

\section*{Declarations}
Conflict of interest: The authors declare that they have no conflict of interest.

\section*{Acknowledgments}

{\color{black}We appreciate the reviewer for taking the time to carefully review the manuscript and give detailed and constructive comments, which has greatly helped to improve this paper. }

Hao would like to thank  Professor Jianlin Xia for helpful discussions during the revision of this manuscript {\color{black}and the start--up grant from Southeast University in China. }
Zhang was partially supported
by the Air Force Office of Scientific Research under award number FA9550-20-1-0056.

\bibliographystyle{plain}
\bibliography{RBF_lap_ref} 
 \end{document}